\documentclass[a4paper,10pt]{article}
\usepackage{graphicx}
\usepackage{amsmath, amsthm, slashed}
\usepackage{pxfonts, verbatim}
\usepackage{rotating}
\usepackage{epsfig}

\usepackage{rotating}

\newtheorem{definition}{Definition}[section]
\newtheorem{theorem}{Theorem}[section]

\newtheorem{corollary}{Corollary}[section]
\newtheorem{proposition}{Proposition}[section]
\newtheorem{lemma}{Lemma}[section]
\newtheorem{remark}{Remark}[section]

\usepackage{fourier}
\usepackage{color}
\newcounter{mnotecount}[section]

\renewcommand{\themnotecount}{\thesection.\arabic{mnotecount}}

\newcommand{\mnote}[1]
{\protect{\stepcounter{mnotecount}}$^{\mbox{\footnotesize
$
\bullet$\themnotecount}}$ \marginpar{
\raggedright\em
$\!\!\!\!\!\!\,\bullet$\themnotecount: #1} }

\author{Jacques Smulevici\footnote{Laboratoire de Math\'ematiques, Universit\'e Paris-Sud 11, b\^at. 425, 91405 Orsay, France.}}

\title{Small data solutions of the Vlasov-Poisson system and the vector field method}

\begin{document}
\maketitle

\begin{abstract}
The aim of this article is to demonstrate how the vector field method of Klainerman can be adapted to the study of transport equations. After an illustration of the method for the free transport operator, we apply the vector field method to the Vlasov-Poisson system in dimension $3$ or greater. The main results are optimal decay estimates and the propagation of global bounds for commuted fields associated with the conservation laws of the free transport operators, under some smallness assumption. Similar decay estimates had been obtained previously by Hwang, Rendall and Vel\'azquez using the method of characteristics, but the results presented here are the first to contain the global bounds for commuted fields and the optimal spatial decay estimates. In dimension $4$ or greater, it suffices to use the standard vector fields commuting with the free transport operator while in dimension $3$, the rate of decay is such that these vector fields would generate a logarithmic loss. Instead, we construct modified vector fields where the modification depends on the solution itself. 

The methods of this paper, being based on commutation vector fields and conservation laws, are applicable in principle to a wide range of systems, including the Einstein-Vlasov and the Vlasov-Nordstr\"om system. 
\end{abstract}
\tableofcontents
\section{Introduction}
A standard approach to the study of asymptotic stability of stationary solutions of non-linear evolution equations consists in an appropriate linearization of the system together\footnote{A third ingredient not needed in the present case is that of modulation theory, see for instance \cite{lmr:ossgm} for an application of modulation theory in the context of the Vlasov-Poisson system.} with
\begin{enumerate}
\item a robust method for proving decay of solutions to the linearized equations,
\item an appropriate set of estimates for the non-linear terms of the original system, using the linear decay estimates obtained previously. 
\end{enumerate}
For systems of non-linear wave equations such as the Einstein vacuum equations $Ric(g)=0$, several methods for proving decay of solutions to the linear wave equation $\square \psi=0$ where $\displaystyle{\square=-\partial_t^2+\sum_{i=1}^n\partial_{x^i}^2}$ is the wave operator of the flat Minkowski space\footnote{In case of perbutations around a non-flat solution with metric $g$, the operator $\square$ would naturally be replaced by $\square_g$, the wave operator of the metric $g$.} are a priori available. One of the classical methods to derive decay estimates is to use an explicit representation of the solutions, such as the Fourier representation, together with specific estimates for singular or oscillatory integrals. While this method provides very precise estimates on the solutions, it does not seem sufficiently robust to be applicable to quasilinear system of wave equations such as the Einstein equations, and the method of choice\footnote{More recently, a mix of microlocal and vector field methods have also been successfully developped, in particular to handle complex geometries involving trapped trajectories, see for instance \cite{drsr:diii} for an application of these tools.} for proving decay in view of such applications is the commutation vector field method of Klainerman \cite{sk:udelicw} and its extensions using multiplier vector fields, see for instance \cite{cm:lap, cm:tdkg, dr:npsdw}. The method of Klainerman is based on 
\begin{enumerate}
\item A coercive conservation law: the standard energy estimate in the case of the wave equation. 
\item Commutation vector fields: these are typically associated with the symmetries of the equations. In the case of the wave equation, these are the Killing and conformal Killing fields of the Minkowski space. 
\item Weighted vector field idendities and weighted Sobolev inequalities: the usual vector fields $\partial_t, \partial_{x^i}$ are rewritten in terms of the commutation vector fields. The coefficients involved in these decompositions contain weights in $t$ and $|x|$ and the presence of these weights leads to weighted Sobolev inequalities, that is to say decay estimates.
\end{enumerate}

The typical method used in the study of the Vlasov-Poisson and other systems of transport equations such as the Vlasov-Nordstr\"om system is the method of characteristics. This is an explicit representation of the solutions and thus, in our opinion, should be compared with the Fourier representation for solutions of the wave equation. What would then be the analogue of the vector field method for transport equations? 
The aim of this article is twofold. First, we will provide a vector field method for the free transport operator. In fact, in a joint work with J. Joudioux and D. Fajman, we have developped a vector field approach to decay of averages not only for the free (non-relativistic) transport operator but also for the massive and massless relativistic transport operators, see \cite{fjs:vfmrgt}. In this paper, we will give two different proofs of Klainerman-Sobolev inequalities. The easier proof will give us a decay estimate for velocity averages of sufficiently regular distribution functions, i.e. quantities such as $\int_{v \in \mathbb{R}^n} f(t,x,v)dv$. However, this proof fails in the case of velocity averages of absolute values of distribution functions, i.e. quantities such as $\int_{v \in \mathbb{R}^n} |f|(t,x,v)dv$, because higher derivatives of $|f|$ will typically not lie in $L^1$ even if $f$ is in some high regularity Sobolev space. On the other hand, the decay estimate obtained via the method of characteristics can be applied equally well to $f$ and $|f|$. We shall therefore give a second proof of Klainerman-Sobolev inequalities for velocity averages which will be applicable to absolute values of regular distribution functions. The first approach, which is closer to the standard proof of the Klainerman-Sobolev inequality for wave equations, consists essentially of two steps, a weighted Sobolev type inequality for functions in $L^1_x$ and an application of this inequality to velocity averages, exploiting the commutation vector fields. The improvement in the second approach comes from mixing the two steps together. 

In the second part of this paper, we will apply our method to the Vlasov-Poisson system in dimension $n \ge 3$
\begin{eqnarray} \label{eq:vp1}
\partial_t f + v.\nabla_x f + \mu \nabla_x \phi . \nabla_v f &=&0, \\
\Delta \phi&=& \rho(f), \label{eq:vp2} \\
f(t=0)&=&f_0.\label{eq:vp3}
\end{eqnarray}
where $\mu=\pm 1$, $\displaystyle{\Delta=-\sum_{i=1}^n \partial^2_{x^i}}$, $f=f(t,x,v)$ with $t \in \mathbb{R}$, $x,v \in \mathbb{R}^n$, $f_0$ is a sufficiently regular function of $x,v$ and $\rho(f)$ is given by 
$$
\rho(f)(t,x):= \int_{v \in \mathbb{R}^n} f(t,x,v) d^n v.
$$

 Our main result can be summarized as follows (a more precise version is given in Section \ref{se:mr}).
 \begin{theorem}\label{th:mrsf}
Let $n \ge 3$ and $N \ge \frac{5n}{2}+2$ if $n \ge 4$ and $N \ge 14$ if $n=3$. Let $0< \delta < \frac{n-2}{n+2}$. Then, there exists $\epsilon_0 > 0$ such that for all $0 \le \epsilon \le \epsilon_0$, if 
$E_{N,\delta}[f_0] \le \epsilon$, where $E_{N,\delta}[f_0]$ is a norm\footnote{See Section \ref{se:tn} for a precise definition of the norms. The $\delta$ encodes some additional integrability properties of the solutions.} containing up to $N$ derivatives of $f_0$, then the classical solution $f(t,x,v)$ of \eqref{eq:vp1}-\eqref{eq:vp3} exists globally in time\footnote{Under some mild conditions on the initial data, global existence is already guaranteed from the works \cite{kp:gcvp, lp:pmrvp}, so the main points of the theorem, apart from providing an illustration of our new method, are the propagation of the global bounds and the optimal space and time decay estimates for the solutions.}  and satisfies the estimates, $\forall t \in \mathbb{R}$ and $\forall x \in \mathbb{R}^n$,
\begin{enumerate}
\item{\emph{Global bounds}}
\begin{equation}\label{ineq:gb}
\quad E_{N,\delta}[f](t) \le 2 \epsilon.
\end{equation}
\item{\emph{Space and time pointwise decay of averages of $\rho(f)$}}
\begin{align} 
\text{for any multi-index\,\,}&\alpha\text{\,\,with\,\,} |\alpha| \le N-n,\nonumber\\
&|\rho( Z^\alpha f)(t,x) | \le \frac{C_{N,n,\delta} \epsilon}{\left(1+|t|+|x| \right)^{n}},\nonumber
\end{align}
where $Z^\alpha$ is a differential operator of order $\alpha$ obtained as a combination of $|\alpha|$ commuting vector fields and $C_{N,n,\delta} > 0$ is a constant depending only on $N,n,\delta$. 
\item{\emph{Improved decay estimates for derivatives of $f$}}
\begin{align}
\text{for any multi-index\,\,}&\alpha\text{\,\,with\,\,} |\alpha| \le N-n,\nonumber\\
&|\rho( \partial_{x}^\alpha f)(t,x) | \le \frac{C_{N,n,\delta} \epsilon}{\left(1+|t|+|x| \right)^{n+|\alpha|}}.\nonumber 
\end{align}
\item{\emph{Boundedness of the $L^{1+\delta}$ norms of $\nabla^2 \phi$ and $\nabla^2 Z^\alpha \phi$}}
$$\text{for any multi-index $\alpha$ with $|\alpha| \le N$}, \quad ||\nabla^2 Z^\alpha \phi(t)  ||_{L^{1+\delta}(\mathbb{R}^n)} \le C_{N,n,\delta} \epsilon. $$
\item{\emph{Space and time decay of the gradient of the potential and its derivatives}}
\begin{align}
\text{for any multi-index $\alpha$ with\,\,}& |\alpha| \le N-(3n/2+1), \nonumber \\
&| \nabla Z^\alpha \phi (t,x) | \le \frac{C_{N,n,\delta} \epsilon}{t^{(n-2)/2}\left(1+|t|+|x| \right)^{n/2}},\nonumber
\end{align}
as well as the improved decay estimates
$$ 
|\partial^\alpha_x \nabla \phi (t,x) | \le \frac{C_{N,n,\delta} \epsilon}{t^{(n-2)/2}\left(1+|t|+|x| \right)^{n/2+|\alpha|}}.
$$
\end{enumerate}
\end{theorem}
\begin{remark}Stronger bounds can be propagated by the equations provided the data enjoy additional integrability conditions. More precisley, the improved decay estimates for derivatives of $\rho(f)$ can be improved to 
$$
|\rho( \partial_t^\tau \partial_{x}^\alpha f)(t,x) | \le \frac{C_N \epsilon}{\left(1+|t|+|x| \right)^{n+|\alpha|+\tau}}
$$
and the improved decay estimates for derivatives of the gradient of $\phi$ can be improved to 
$$
|\partial^\tau_t\partial^\alpha_x \nabla \phi (t,x) | \le \frac{C_N \epsilon}{t^{(n-2)/2}\left(1+|t|+|x| \right)^{n/2+|\alpha|+\tau}},
$$
the point being that additional $t$ derivatives now bring additional decay in $t$ and $|x|$. These stronger estimates hold provided the initial data have stronger decay in $x,v$ than what is needed in the proof of Theorem \ref{th:mrsf}. Similarly, one can propagates $L^p$ norms with $p \ge 2$ for $\nabla \phi$ and $\nabla Z^\alpha \phi$ provided additional $v$ decay of the initial data is assumed.
\end{remark}
\begin{remark}
Similar time decay estimates have been obtained in \cite{hrv:ogeabvp} for derivatives of $\rho(f)$ and $\phi$ using the method of characteristics under different assumptions on the initial data. On the other hand, the optimal decay rates in space and the propagation of the global bounds \eqref{ineq:gb} were, as far as we know, not known prior to our work. 
\end{remark}

\begin{remark}
As is clear from the proof below and is typical of strategies based on commutation formulae and conservation laws, the method is very robust. In particular, we are not using the method of characteristics, nor the conservation of the total energy for the system \eqref{eq:vp1}-\eqref{eq:vp3}. An illustration of this robustness will be given in \cite{fjs:vfmrgt} where we will apply a similar approach to the study of the Vlasov-Nordstr\"om system.
\end{remark}

\subsection*{Previous work on the Vlasov-Poisson system and discussion}
There exists a large litterature on the Vlasov-Poisson system. We refer to the introduction in \cite{cm:sb} for a good introduction to the subject and only quote here the most important results from the point of view of this article. In the pioneered work \cite{bd:gevp}, small data global existence in dimension $3$ for the Vlasov-Poisson system was established together with optimal time decay rates for $\rho(f)$ and $\nabla \phi$ but no decay was obtained for their derivatives. The optimal time (but not spatial) decay rates for derivatives of $\rho(f)$ and $\nabla \phi$ has been only much later obtained in \cite{hrv:ogeabvp}, covering at the same time all dimensions $n \ge 3$. Both these works use decay estimates obtained via the method of characteristics. In fact, in \cite{bd:gevp} and even more in \cite{hrv:ogeabvp}, precise estimates on the deviation of the characteristics from the characteristics of the free transport operator are needed in order to obtain the desired decay estimates. Parallely to these works giving information on the asymptotics of small data solutions, let us mention that under fairly weak assumption on the initial data (in particular, no smallness assumption is needed), it is known that global existence holds in dimension $3$ for the solutions of \eqref{eq:vp1}-\eqref{eq:vp3}, see \cite{kp:gcvp, lp:pmrvp}. The strongest results concerning the stability of non-trivial stationnary solutions of \eqref{eq:vp1}-\eqref{eq:vp3} with $\mu=-1$ have been obtained in \cite{lmr:ossgm}. They are not based on decay estimates but on a variational characterisation of the stationary solutions. On the other hand, this type of method does not provide asymptotic stability of the solutions but orbital stability. It is likely that any result addressing the question of asymptotic stability will need to go back to an appropriate linearization of the equations combined with robust decay estimates\footnote{See for instance \cite{gr:nvanssd} for some stability results using the linearization approach in the case of the spherically-symmetric King model.}. We believe that, once again, the vector field method would be totally appropriate for the derivation of such decay estimates. Finally, let us mention the celebrated work \cite{mv:ld} on Landau damping concerning the stability of stationnary solutions to \eqref{eq:vp1}-\eqref{eq:vp3} with periodic initial data. In view of the present work, it will be interesting to try to revisit this question using vector field methods. 

\subsection*{Outline of the paper}
In Section \ref{se:pre}, we introduce the vector fields commuting with the free transport operator and the notations that we will use throughout the paper. In Section \ref{se:dafovfm}, we present and prove decay estimates for velocity averages of solutions to the transport equation. In the following section, we present our results on the Vlasov-Poisson system. The remaining last two sections are devoted to the proof of these results, first in dimension $n \ge 4$ and then in dimension $3$ using modified vector fields. 
\subsection*{Acknowledgements}
This project was motivated by my joint work with David Fajman and J\'er\'emie Joudioux on relativistic transport equations. I would like to thank both of them, as well as Christophe Pallard and Fr\'ed\'eric Rousset for many interesting discussions on these topics. I would also like to thank Pierre Rapha\"el for an extremely stimulating conversation which took place during the conference ``Asymptotic analysis of dispersive partial differential equations'' held in October 2014 at Pienza, Italy. Some of this research was done during this conference. Finally, I would like to acknowledge partial funding from the Agence Nationale de la Recherche ANR-12-BS01-012-01 (AARG) and ANR SIMI-1-003-01.
 
\section{Preliminaries} \label{se:pre}
Throughout this article, $f$ will denote a sufficiently regular function of $(t,x,v)$ with $t \in \mathbb{R}$ and $(x,v) \in \mathbb{R}^n \times \mathbb{R}^n$. 
By sufficiently regular, we essentially mean that $f$ is such that all the terms appearing in the equations make sense as distributions and that all the norms appearing in the estimates are finite. For simplicity, the reader might just assume that $f$ is smooth with compact support in $x,v$ (but any sufficient fall-off will be enough). 

We will denote by $T$ the free transport operator i.e. 

$$
T(f):=\partial_tf + \sum_{i=1}^n v^i \partial_{x^i} f,
$$
where $\partial_t f= \frac{\partial f}{\partial t}$ and for all $1\le i \le n$, $\partial_{x^i}f= \frac{\partial f}{\partial x^i}$.
Similarly, for any sufficiently regular scalar function $\phi$, $T_\phi$ will denote the perturbed transport operator

\begin{equation} \label{def:tphiop}
T_\phi (f):= T(f) + \mu \nabla_x \phi. \nabla_v f,
\end{equation}
where $\mu=\pm 1$, corresponding to an attractive or repulsive force. Since we are dealing only with small data solutions, the sign of $\mu$ will play no role in the rest of this article.

The notation $A \lesssim B$ will be used to specify that there exists a universal constant $C>0$ such that $A \le CB$, where typically $C$ will depends only on the number of dimensions $n$ and a few other fixed constants, such as the maximum number of commutations. 
\subsection{Macroscopic and microscopic vector fields}\label{se:mmvf}
Consider first the following set of vector fields 

\begin{itemize}
\item{Translations in space and time}\,\, $\partial_t$, $\partial_{x^i}$,
\item{Uniform motion in one spatial direction\footnote{Recall that these vector fields are the generators of the Gallilean transformations of the form $x \in \mathbb{R}^n \rightarrow x+tv_i$, where $v_i^k=\delta_i^k$.}} \,\, $t\, \partial_{x^i}$,
\item{Rotations}\,\, $x^i \partial_{x^j}-x^j \partial_{x^i}$,
\item{Scaling in space}\,\,  $\displaystyle{\sum_{i=1}^n x^i \partial_{x^i}}$,
\item{Scaling in space and time}\,\, $t \partial_t + \displaystyle{\sum_{i=1}^n x^i \partial_{x^i}}$.
\end{itemize}
The above set of vector fields is associated with the Gallilean invariance of macroscopic fields and equations. We will denote by $\Gamma$ the set of all such vector fields
$$
\Gamma =\left\{\partial_t, \partial_{x^i}, t\, \partial_{x^i}, x^i \partial_{x^j}-x^j \partial_{x^i}, \sum_{i=1}^n x^i \partial_{x^i}, t \partial_t + \sum_{i=1}^n x^i \partial_{x^i}, \,\, 1\le i,j \le n\right\}.
$$


One easily check that while the translations commute with $T$, uniform motions, rotations or the scaling in space do not. The correct replacement for these vector fields is most easily explained using the language of differential geometry; the interested reader may consult \cite{Sarbach:2013vy, fjs:vfmrgt} for detailed constructions (in the case of the relativistic transport operator). We shall here only present the resulting objects which are the vector fields

\begin{itemize}
\item{Uniform motions in one direction in microscopic form} $t \partial_{x^i} + \partial_{v^i}$,
\item{Rotations in microscopic form} $x^i \partial_{x^j}-x^j \partial_{x^i}+ v^i \partial_{v^j}-v^j \partial_{v^i}$,
\item{Scaling in space in microscopic form} $\displaystyle{\sum_{i=1}^n x^i \partial_{x^i}+v^i\partial_{v^i}}$.
\end{itemize}
One can then easily check
\begin{lemma}[Commutation with the transport operator]\hspace{2cm}\label{lem:cp}
\begin{itemize}
\item If $Z$ is any of the translations, microscopic uniform motions or microscopic rotations, then $[T,Z]=0$.
\item If $Z$ is the microscopic scaling in space, then $[T,Z]=0$. 
\item If $Z$ is the scaling in space and time, then $[T,Z]=T$.
\end{itemize}
\end{lemma}

\begin{remark} \label{rk:acvf}
From the two scaling commuting vector fields, it follows automatically that $t\partial_t - \sum_{i=1}^n v^i\partial_{v^i}$ also commutes with $T$ in the sense that 
$$
[ T, t\partial_t - \sum_{i=1}^n v^i\partial_{v^i} ]=T.
$$
This vector field will be used to obtain improved decay for $t$ derivatives of velocity averages.
\end{remark}



To ease the notation, we will denote by $\Omega_{ij}^{x}:=x^i \partial_{x^j}-x^j \partial_{x^i}$ the rotation vector fields in $x$ and by $\Omega_{ij}^{v}:=v^i \partial_{v^j}-v^j \partial_{v^i}$ the rotation vector fields in $v$. The full microscopic rotation vector fields are thus of the form $\Omega_{ij}^{x}+\Omega_{ij}^{v}$. Similarly, we will denote by $S^{x}+S^{v}:=\displaystyle{\sum_{i=1}^n x^i \partial_{x^i}+v^i\partial_{v^i}}$ the scaling in space in microscopic form, with $S^{x}=\displaystyle{\sum_{i=1}^n x^i \partial_{x^i}}$ and $S^{v}=\displaystyle{\sum_{i=1}^n v^i\partial_{v^i}}$.

Let now $\gamma$ be the set of all the above microscopic vector fields including the translations and the space and time scaling i.e. 

$$
\gamma=\left\{ \partial_t, \partial_{x^i}, t\partial_{x^i} + \partial_{v^i}, \Omega_{ij}^{x}+\Omega_{ij}^{v}, S^{x}+S^{v} , t \partial_t + \sum_{i=1}^n x^i \partial_{x^i}, 1\le i,j \le n \right\}.
$$

\subsection{Multi-index notations}

Let $Z^i,i=1,..,2n+3+n(n-1)/2$ be an ordering of $\Gamma$. For any multi-index $\alpha$, we will denote by $Z^\alpha$ the differential operator of order $|\alpha|$ given by the composition $Z^{\alpha_1} Z^{\alpha_2}..$. 

In view of the above discussion, to any vector field of $\Gamma$, we can associate a unique vector field of $\gamma$. More precisely, 
\begin{eqnarray*}
\partial_t &\rightarrow& \partial_t, \\
\partial_{x^i} &\rightarrow& \partial_{x^i},\\
t \partial_{x^i} &\rightarrow& t \partial_{x^i} + \partial_{v^i}, \\
\Omega_{ij}^x &\rightarrow& \Omega_{ij}^x+\Omega_{ij}^v, \\
S^x &\rightarrow& S^x+ S^v, \\
t\partial_t + \sum_{i=1}^n x^i \partial_{x^i} &\rightarrow& t\partial_t + \sum_{i=1}^n x^i \partial_{x^i}.
\end{eqnarray*}
Thus, to any ordering of $\Gamma$ we can associate an ordering of $\gamma$. We will by a small abuse of notation, denote again by $Z^i$ the elements of such an ordering since it will be clear that, if $Z^i$ is applied to a macroscopic quantity, such as a velocity average, then $Z^i \in \Gamma$ and if $Z^i$ is applied to a microscopic quantity, i.e.~any function depending on $(x,v)$ (and possibly $t$), then $Z^i \in \gamma$. Similarly, for any multi-index $\alpha$, we will also denote by $Z^\alpha$ the differential operator of order $|\alpha|$ given by the composition $Z^{\alpha_1} Z^{\alpha_2}..$ obtained from the vector fields of $\gamma$.

For some of the estimates below, it will be sufficient to only consider a subset of all the vector fields of $\Gamma$ and $\gamma$. 
Let us thus denote by $\Gamma_s$ the set of all the macroscopic vector fields apart from $\partial_t$ and $t \partial_t + \displaystyle{\sum_{i=1}^n x^i \partial_{x^i}}$, which are the only vector fields containing time derivatives and by $\gamma_s$ the corresponding set of microscopic vector fields, i.e. 
\begin{eqnarray*}
\Gamma_s&=&\left\{\partial_{x^i}, t\, \partial_{x^i}, x^i \partial_{x^j}-x^j \partial_{x^i}, \sum_{i=1}^n x^i \partial_{x^i},  \, 1\le i,j \le n\right\},\\
\gamma&=&\left\{\partial_{x^i}, t\partial_{x^i} + \partial_{v^i}, \Omega_{ij}^{x}+\Omega_{ij}^{v}, S^{x}+S^{v} , \, 1\le i,j \le n \right\}.
\end{eqnarray*}

The notation $Z^\alpha \in \Gamma^{|\alpha|}_s$ (respectively $Z^\alpha \in \gamma^{|\alpha|}_s$) will be used to denote a generic differential operator of order $|\alpha|$ obtained as a composition of $|\alpha|$ vector fields in $\Gamma_s$ (respectively in $\gamma_s$). The standard notation $\partial_x^\alpha$ will also be used to denote a differential operator of order $|\alpha|$ obtained as a composition of $|\alpha|$ translations among the $\partial_{x^i}$ vector fields. 


The following lemmae can easily be checked.

\begin{lemma}[Commutation within $\Gamma$]
For any $Z^\alpha \in \gamma^{|\alpha|}$, $Z^{\alpha'} \in \gamma^{|\alpha'|}$ , where $\alpha, \alpha'$ are multindices, we have $$[Z^\alpha, Z^{\alpha'}]=\sum_{|\beta|\le |\alpha|+|\alpha'|-1} c_\beta^{\alpha,\alpha'} Z^\beta,$$ for some constant coefficients $c_\beta^{\alpha, \alpha'}$. 
Moreover, if $Z^\alpha, Z^{\alpha'} \in \gamma_s$, then all the $Z^\beta$ of the right-hand side belongs to $\gamma_s$.
\end{lemma}

\begin{lemma}[Commutation of $Z^\alpha$ and weights in $v$] \label{lem:czv}
Let $q>0$. For any sufficienly regular function $f$ of $(t,x,v)$ and for any  $Z^\alpha \in \gamma^{|\alpha|}$ where $\alpha$ is a multi-index, we have
$$
\left| Z^\alpha \left[ (1+v^2)^{q/2} f\right] \right| \lesssim (1+v^2)^{q/2} \sum_{|\beta|  \le |\alpha|} \left| Z^\beta(f) \right|.
$$
Moreover, if $Z^\alpha \in \gamma_s^{|\alpha|}$ then all the $Z^\beta$ belong to $\gamma_s^{|\beta|}$ in the above inequality.
\end{lemma}

\subsection{Velocity averages and commutators}
For any integrable function $f$ of $v \in \mathbb{R}^n$, we will denote by $\rho(f)$ the quantity

$$
\rho(f):= \int_{v \in \mathbb{R}^n} f d^nv
$$

We have the following lemma. 

\begin{lemma} \label{lem:vac}
For any sufficiently regular function $f$ of $(x,v)$ we have
\begin{itemize}
\item for all $1 \le i \le n$,
$$\partial_{x^i} \rho(f)= \rho( \partial_{x^i} f),$$

\item for all $t \in \mathbb{R}$ and all $1 \le i \le n$, $$t \partial_{x^i}\rho(f)=\rho\left(\left( t \partial_{x^i}+ \partial_{v^i} \right) f \right),$$
\item for all $1 \le i, j \le n,$
$$\Omega_{ij}^{x} \rho(f)= \rho\left( \left(\Omega_{ij}^{x}+\Omega_{ij}^{v} \right)(f) \right),$$ 
\item for all $t \in \mathbb{R}$ and $x \in \mathbb{R}^n$, 
$$
\left(t \partial_t +\sum_{i=1}^n x^i \partial_{x^i}\right) \rho(f)= \rho\left( \left( t \partial_t +\sum_{i=1}^n x^i \partial_{x^i}\right)(f) \right),
$$
\item and finally $$S^{x}\rho(f)=\rho\left( \left(S^{x}+S^{v} \right)(f) \right)+n\rho(f),$$
where $S^x$ and $S^x+S^v$ are the spatial scaling vector fields in macroscopic and  microscopic forms.
\end{itemize}
\end{lemma}
\begin{proof}
The proof is straigtforward and consist in identifying total derivatives in $v$. For instance, we have 
\begin{eqnarray*}
\rho\left( S^{v} \right)(f)=\int_{v \in \mathbb{R}^n} \left(\sum_{i=1}^n v^i \partial_{v^i} f\right)d^nv=\int_{v \in \mathbb{R}^n} -n f d^n v,
\end{eqnarray*}
where we have integrated by parts in each of the $v^i$.
Similarly, in the case of rotations, it suffices to note that for any $1 \le i, j \le n$, $\Omega^v_{ij}$ is an angular derivative in $v$ and thefore, $\int_{v \in\mathbb{R}^n} \Omega_{ij}^v(f)d^nv=0.$
\end{proof}

In the remainder of this paper, we shall write the preceding lemma as 

$$
Z\rho(f)=\rho(Z(f))+c_Z \rho(f),
$$
where we are using, by a small abuse of notation, the letter $Z$ to denote a generic macroscopic vector field and its corresponding microscopic version and where $c_Z=0$ unless $Z$ is the spatial scaling vector field, in which case $c_Z=n$. 

\subsection{Vector field identities}
The following well-known identity will be used later

\begin{lemma} \label{lem:ivfs}
For any $1\le j \le n$, we have
\begin{equation} \label{id:ivfs}
|x|^2 \partial_{x^j} = \sum_{i=1}^n x^i \Omega_{ij}^{x}+ x^j S^{x},
\end{equation}
and thus, at any $x \neq 0$, 
 $$|x|\partial_{x^j}=\sum_{i=1}^n \frac{x^i}{|x|} \Omega_{ij}^{x}+\frac{x^j}{|x|}S^{x},$$ where the coefficients $\frac{x^i}{|x|}$ are all homogeneous of degree $0$ and therefore uniformly bounded.
\end{lemma}
The following higher order version will be used often in the derivation of the Klainerman-Sobolev inequalities of the next section. 
\begin{lemma} \label{lem:ttks}
For any multi-index $\alpha$, 
$$
(t+|x|)^\alpha \partial_x^\alpha= \sum_{|\beta| \le |\alpha|, Z^\beta \in \Gamma^{|\beta|}_s} C_\beta Z^\beta,
$$
where the coefficients $C_\beta$ are all uniformly bounded. 
\end{lemma}
\begin{proof}
The lemma is a consequence the previous decomposition, the fact $t \partial_{x^i}$ is part of our algebra of commuted vector fields and that $\left[\partial_{x^i},t+|x| \right]$ is homogeneous of degree $0$.
\end{proof}
\subsection{The commuted equations}
We now turn to the study of the transport operator $T_\phi$ defined by \eqref{def:tphiop}. 
Many of the estimates below are only valid provided $\phi$ has sufficient regularity. In the applications to the Vlasov-Poisson system of this article, we will eventually control the regularity of $\phi$ via a bootstrap argument. For all the estimates below, we therefore assume that $\phi$ is a sufficiently regular\footnote{For instance, one can assume that $\phi$ is a smooth function on $[0,T] \times \mathbb{R}^n_x$ with compact support in $x$.} function of $(t,x)$ defined on $[0,T] \times \mathbb{R}^n_x$, for some $T > 0$, which decays sufficiently fast as $|x| \rightarrow +\infty$.

The following lemma can then easily be checked.
\begin{lemma}\label{lem:comf}
Let $f$ be a sufficiently regular function of (t,x,v) and let $\alpha$ be a multi-index. Then, there exists constant coefficients $C^\alpha_{\beta \gamma}$ such that,
$$
[T_\phi,Z^\alpha] f = \sum_{|\beta| \le |\alpha|-1,} \sum_{|\gamma|+|\beta| \le |\alpha|} C^\alpha_{\beta \gamma}\nabla_x Z^\gamma \phi. \nabla_v Z^\beta f.
$$
Moreover if $Z^\alpha \in \Gamma_s^{|\alpha|}$, then $Z^\gamma \in \Gamma_s^{|\gamma|}$ and $Z^\beta \in \gamma_s^{|\beta|}$ in the above decomposition.
\end{lemma} 
Similarly one has
\begin{lemma} Let $f$ be a sufficiently regular function of $(t,x,v)$ and let $\psi$ be the solution to the Poisson equation $\Delta \psi= \rho(f)$. Then, for any multi-index $\alpha$, $Z^\alpha \psi$ is solution to an equation of the form
\begin{equation}\label{eq:cp}
\Delta Z^\alpha \psi= \sum_{|\beta| \le |\alpha|} C_{\beta}^\alpha Z^\beta \rho(f),
\end{equation}
where $C_{\beta}^\alpha$ are constants. 
\end{lemma}
\begin{proof}
The lemma is an easy consequence of the fact that all macroscopic vector fields apart from the two scalings commute with $\Delta$ while for the spatial scaling $S^{x}$ and the space-time scaling $t\partial_t+ S^{x}$ we have $[\Delta, S^{x}]=2 \Delta$ and $[\Delta, t\partial_t+ S^{x}]=2 \Delta$.

\end{proof}

\subsection{Conservation laws}

We shall use the following (approximate) conservation laws.

\begin{lemma} \label{lem:cl}
For any sufficiently regular function $f$ of $(t,x,v)$, we have, for all $t \in [0,T]$, 
$$
|| f(t) ||_{L^1( \mathbb{R}^n_x \times \mathbb{R}^n_v)} \le || f(0) ||_{L^1( \mathbb{R}^n_x \times \mathbb{R}^n_v)}+\int_0^t || T_\phi (f)(s) ||_{L^1( \mathbb{R}^n_x \times \mathbb{R}^n_v )}ds.
$$
Similarly, we have for all $p \ge 1$, for all $t \in [0,T]$, 
$$
|| f(t) ||^p_{L^p( \mathbb{R}^n_x \times \mathbb{R}^n_v)}\le || f(0) ||^p_{L^p( \mathbb{R}^n_x \times \mathbb{R}^n_v)}+p \int_0^t ||  f^{p-1} T_\phi (f)(s) ||_{L^1( \mathbb{R}^n_x \times \mathbb{R}^n_v )}ds,
$$
and for all $q\ge 1$,
\begin{eqnarray}
||(1+v^2)^{q/2} f^p(t) ||_{L^1( \mathbb{R}^n_x \times \mathbb{R}^n_v)}&\lesssim& ||(1+v^2)^{q/2} f^p(0) ||_{L^1( \mathbb{R}^n_x \times \mathbb{R}^n_v)} \nonumber\\
&&\hbox{}+\int_0^t || (1+v^2)^{(q-1)/2} f^{p-1} T_\phi(f)(s) ||_{L^1( \mathbb{R}^n_x \times \mathbb{R}^n_v )} ds \nonumber\\
&&\hbox{}+\int_0^t ||  f^{p}  (1+v^2)^{(q-1)/2} \partial_x \phi(s) ||_{L^1( \mathbb{R}^n_x \times \mathbb{R}^n_v )}ds . \label{eq:Ewv}
\end{eqnarray}
Note in particular that the conclusions of the lemma hold true when $T_\phi=T$ (i.e. when $\partial_x \phi=0$).
\end{lemma}
\begin{proof}These are classical estimates so we only sketch their proofs.

One has (in the sense of distribution), 
$$
T_\phi \left[ (1+v^2)^{q/2} |f|^p \right] \lesssim (1+v^2)^{q/2} |f|^{p-1} |T_\phi(f)|+ (1+v^2)^{(q-1)/2}|\partial_x \phi| |f|^p.
$$
Using a standard procedure\footnote{For instance, assume first that $f$ has compact support in $(x,v)$ with a uniform bound on the support of $f$ in $(x,v)$ for $t \in [0,T]$. For all $\epsilon > 0$, consider the function $f_\epsilon=\sqrt{\epsilon^2+f^2}\chi(x,v)$, where $\chi(x,v)$ is a smooth cut-off function which is $1$ on the support of $f$ and vanishes for large $x$ and $v$. Apply then the previous estimates to $f_\epsilon$ and take the limit $\epsilon \rightarrow 0$. A standard density argument deals with the case of non-compact support.}, one can regularize the previous inequality. We shall therefore neglect regularity issues here. Integrating the previous line in $(t,x,v)$ leads to 
\begin{eqnarray}\label{es:evp}
&&\int_0^t \int_x \int_v T_\phi \left[ (1+v^2)^{q/2} |f|^p \right] dv dx ds \lesssim \nonumber \\
&& \quad \quad  \, \int_0^t \int_x \int_v \left[ (1+v^2)^{q/2} |f|^{p-1} |T_\phi(f)|+ (1+v^2)^{(q-1)/2}|\partial_x \phi| |f|^p \right] dv dx ds.
\end{eqnarray}
On the other-hand, remembering that $T_\phi=\partial_t+\sum_{i=1}^n v^i \partial_x^i+ \mu \nabla_x \phi. \nabla_v$ and integrating by parts in $x$ and $v$, we obtain
\begin{eqnarray*}
\int_0^t \int_x \int_v T_\phi \left[ (1+v^2)^{q/2} |f|^p\right] dv dx ds&=&||(1+v^2)^{q/2} f^p(t)||_{L^1(\mathbb{R}_x^n \times \mathbb{R}_v^n)}\\
&&\hbox{}-||(1+v^2)^{q/2} f^p(0)||_{L^1(\mathbb{R}_x^n \times \mathbb{R}_v^n)},
\end{eqnarray*}
which combined with \eqref{es:evp} leads to the desired estimate \eqref{eq:Ewv}.

\end{proof}

\section{Decay of velocity averages for the free transport operator via the vector field method} \label{se:dafovfm}
Since the main purpose of this article is to illustrate how the vector field method can lead to robust decay estimates for velocity averages, let us for the sake of comparison recall the Bardos-Degond decay estimate and its proof. 
\begin{proposition}[\cite{bd:gevp}]
Let $f$ be a sufficiently regular solution of $T(f)=0$. Then, we have the estimate, for all $t > 0$ and all $x \in \mathbb{R}^n$, 
\begin{equation}\label{es:bd}
|\rho(f)|(t,x) \le \frac{1}{t^n} \int_{x \in \mathbb{R}^n} \sup_{v \in \mathbb{R}^n} |f(0,x,v)| d^n x.
\end{equation}
\end{proposition}

\begin{proof}
The proof of this classical estimate is based on the method of charateristics. 
More precisly, if $f$ is a regular solution to $T(f)=0$, then it follows that
$$
f(t,x,v)=f(0,x-vt,v),
$$
for all $t \in \mathbb{R}$, $x,v \in \mathbb{R}^n$. 

We then have
\begin{eqnarray*}
|\rho(f)|(t,x)&=& \left| \int_{v \in \mathbb{R}^n} f(t,x,v) d^n v \right|\\
&=&\left| \int_{v \in \mathbb{R}^n} f(0,x-vt,v) d^n v \right|\\
&\le&  \int_{v \in \mathbb{R}^n} \sup_{w \in \mathbb{R}^n} |f(0,x-vt,w)| d^n v.
\end{eqnarray*}
Applying now the change of coordinates $y=vt$ for $t> 0$ leads to 
\begin{eqnarray*}
|\rho(f)|(t,x)&\le &  \frac{1}{t^n} \int_{y \in \mathbb{R}^n} \sup_{w \in \mathbb{R}^n} |f(0,x-y,w)| d^n y=\frac{1}{t^n} \int_{x \in \mathbb{R}^n} \sup_{v \in \mathbb{R}^n} |f(0,x,v)| d^n x .
\end{eqnarray*}
\end{proof}
In the above proof, the two key ingredients are 
\begin{itemize}
\item the explicit representation obtained via the method of characteristics,
\item the change of variables $y=vt$. 
\end{itemize}
Note that in the presence of a perturbation of the free transport operator, to exploit a similar change of variables would require estimates on the Jacobian associated with the differential of the characteristic flow, see \cite{bd:gevp, hrv:ogeabvp}.

Let us now show how the vector field method can be used as an alternative to obtain similar decay estimates. As explained in the introduction, we will give two different proofs. 

The first proof will give us decay estimates for quantities of the form $\int_{v \in \mathbb{R}^n} f dv$. The starting point of this approach is the following Klainerman-Sobolev inequality using $L^1(\mathbb{R}^n)$ norms of commuted fields.
\begin{lemma}[$L^1$ Klainerman-Sobolev inequality] \label{lem:ksi}
For any sufficiently regular function $\psi$ defined on $\mathbb{R}^n_x$, we have
$$
|\psi|(x) \lesssim \frac{1}{\left(1+t+|x| \right)^n} \sum_{|\alpha| \le n, Z^\alpha \in \Gamma_s} ||  Z^\alpha (\psi) ||_{L^1( \mathbb{R}^n_x)}.
$$
\end{lemma}
\begin{proof}
This is relatively standard material and we adapt here the presentation given in \cite{cs:lnwe} Chap.2 to our setting.

Fix $(t,x) \in \mathbb{R}_t \times \mathbb{R}_x^n$ and let $\widetilde{\psi}$ be the function 
\begin{eqnarray} \label{eq:psit}
\widetilde{\psi}:\, \mathbb{R}^n &\rightarrow& \mathbb{R} \\
y &\rightarrow& \widetilde{\psi}(y):=\psi(x+(t+|x|)y).
\end{eqnarray}
Applying a standard\footnote{Recall that, while the general $W^{k,p}(\mathbb{R}^n)  \hookrightarrow L^\infty(\mathbb{R}^n)$ embedding requires $k > \frac{n}{p}$, the special case $p=1$ only needs $k \ge n$.} Sobolev inequality, we have
\begin{equation}\label{es:sl1}
|\psi(t,x)| = |\widetilde{\psi}(0) | \lesssim \sum_{|\alpha| \le n} ||\partial_y^\alpha \widetilde{\psi} ||_{L^1(B_n(0,1/2))},
\end{equation}
where $B_n(0,1/2)$ denote the ball in $\mathbb{R}^n_y$ of radius $1/2$.
On the other hand, we have $\partial_{y^i}\widetilde{\psi}(y)= (t+|x|) \partial_{x^i} \psi(x+(t+|x|)y)$ and thus, for $y \in B_n(0,1/2)$, 
\begin{eqnarray*}
|\partial_y^\alpha \widetilde{\psi}(y)| &\lesssim& \sum_{|\beta| \le |\alpha|} (t+|x|)^\beta \left|\partial_x^\beta \psi\left(x+(t+|x|)y\right)\right|, \\
&\lesssim&\sum_{|\beta| \le |\alpha|} (t+|x+(t+|x|)y|)^\beta \left|\partial_x^\beta \psi\left(x+(t+|x|)y\right)\right|,\\
&\lesssim&\sum_{|\beta| \le |\alpha|, Z^\beta \in \Gamma_s}  \left|Z^\beta \psi\left(x+(t+|x|)y\right)\right|,
\end{eqnarray*}
where we have used Lemma \ref{lem:ttks} in the last step and the fact that $t+|x|$ and $t+\left|x+(t+|x|)y\right|$ are comparable for $y \in B_n(0,1/2)$. Inserting the last line in the Sobolev inequality \eqref{es:sl1} and applying the change of variables $z=(t+|x|)y$ conclude the proof of the lemma.

\end{proof}
Using Lemma \ref{lem:vac}, we now note that for any vector field $Z$ and any sufficiently regular function $f$ of $x,v$, we have
$$ || Z (\rho(f)) ||_{L^1(\mathbb{R}^n_x)} = ||  \rho(Z(f))+c_Z\rho(f) ||_{L^1(\mathbb{R}^n_x)}$$
and thus,  
$$
|| Z (\rho(f)) ||_{L^1(\mathbb{R}^n_x)} \lesssim ||  Z(f) ||_{L^1\left(\mathbb{R}^n_x \times \mathbb{R}^n_v\right)}+|| \rho(f)||_{L^1\left(\mathbb{R}^n_x \times \mathbb{R}^n_v\right)}.
$$
Combined with the previous Klainerman-Sobolev inequality, we obtain

\begin{proposition}[Global Klainerman-Sobolev inequality for velocity averages]

For any sufficiently regular function $f$ defined on $\mathbb{R}^n_x \times \mathbb{R}^n_v$, we have, for all $t > 0$ and all $x \in \mathbb{R}^n$, 
\begin{equation} \label{es:gksiva0}
|\rho(f)|(x) \lesssim \frac{1}{\left(1+t+|x| \right)^n} \sum_{|\alpha| \le n, Z^\alpha \in \gamma_s^{|\alpha|}} ||  Z^\alpha f  ||_{L^1( \mathbb{R}_x \times \mathbb{R}_v )}.
\end{equation}
\end{proposition}
Note that the above inequality cannot be apply to $|f|$ even is $f$ is say a smooth, compactly supported function. Indeed, if $|f|$ is say in $W^{n,p}$, then $|f|$ is in $W^{1,p}$ but, unless $f$ has some extra special properties, $|f| \notin W^{n,p}$ when $n \ge 2$. On the other hand, \eqref{es:bd} clearly holds both for $f$ and for $|f|$.

Two disctinct steps lead to the proof of \eqref{es:gksiva0}, the $L^1$ Klainerman-Sobolev inequality of Lemma \ref{lem:ksi} and the special commutation properties of the velocity averaging operator as described in Lemma \ref{lem:vac}. To improve upon \eqref{es:gksiva0}, the strategy is to try to use at the same time arguments similar to those of Lemma \ref{lem:ksi} and Lemma \ref{lem:vac}, instead of applying them one after the other. This will lead to us to the following improvement. 

\begin{proposition}[Global Klainerman-Sobolev inequality for velocity averages of absolute values]
 \label{prop:gksiva1}
For any sufficiently regular function $f$ defined on $\mathbb{R}^n_x \times \mathbb{R}^n_v$, we have, for all $t > 0$ and all $x \in \mathbb{R}^n$, 
\begin{equation}
|\rho(|f|)|(x) \lesssim \frac{1}{\left(1+t+|x| \right)^n} \sum_{|\alpha| \le n, Z^\alpha \in \gamma_s^{|\alpha|}} ||  Z^\alpha f  ||_{L^1( \mathbb{R}_x \times \mathbb{R}_v )}. \label{es:gksiva1}
\end{equation}
\end{proposition}
\begin{proof}
Let us assume that $f$ is smooth and compactly supported for simplicity. 

Define $\widetilde{\psi}$ similarly to \eqref{eq:psit} as
\begin{eqnarray*}
\widetilde{\psi}:\, B_n(0,1/2) &\rightarrow& \mathbb{R} \\
y &\rightarrow& \widetilde{\psi}(y):=\int_{v \in \mathbb{R}^n}|f|(x+(t+|x|)y,v)d^n v.
\end{eqnarray*}
Next, recall that for any $\psi \in W^{1,1}$, $|\psi| \in W^{1,1}$ and $\partial |\psi|= \frac{\psi}{|\psi|} \partial \psi$ (in the sense of distribution, see for instance \cite{ll:a}, Chap 6.17) so that in particular $\left| \partial |\psi| \right| \le  |\partial \psi|$.
Let us write $y=(y_1,..,y_n)$ and let $\delta= \frac{1}{4n}$. Using a $1$ dimensional Sobolev inequality, we have 
\begin{equation}\label{eq:rp1}
|\widetilde{\psi}(0)|\lesssim \int_{|y_1| \le \delta^{1/2}}\left( \left| \partial_{y_1}\widetilde{\psi}(y_1,0,..,0)\right| + \left|\widetilde{\psi}(y_1,0,..,0)\right| \right)dy_1.
\end{equation}
Now, 
\begin{eqnarray}
\partial_{y_1}\widetilde{\psi}(y_1,0,..,0)&=&\int_{v \in \mathbb{R}^n} (t+|x|)\partial_{x_1}|f|\left(x+(t+|x|)(y_1,0,..,0)\right)d^n v, \nonumber \\
&=&\int_{v \in \mathbb{R}^n} t\partial_{x_1}|f|\left(t,x+(t+|x|)(y_1,0,..,0)\right)d^n v \nonumber \\
&&\hbox{}+\int_{v \in \mathbb{R}^n} |x|\partial_{x_1}|f|\left(x+(t+|x|)(y_1,0,..,0)\right)d^n v. \label{eq:sts}
\end{eqnarray}
As before, we can introduce total derivatives in $v$. For instance, 
\begin{eqnarray*}
&&\int_{v \in \mathbb{R}^n} t\partial_{x_1}|f|\left(x+(t+|x|)(y_1,0,..,0)\right)d^n v= \\
&&\hbox{}\quad \quad \int_{v \in \mathbb{R}^n}\left(t\partial_{x_1}+\partial_{v_1} \right)|f|\left(x+(t+|x|)(y_1,0,..,0)\right)d^n v, \\
&&\le \int_{v \in \mathbb{R}^n}\left| \left(t\partial_{x_1}+\partial_{v_1} \right)f\left(t,x+(t+|x|)(y_1',0,..,0)\right) \right|d^n v,
\end{eqnarray*}
and, using that $(y_1,0,..,0) \in B_n(0,1/2)$ if $|y_1| \le \delta^{1/2}$, a similar argument can be used to handle the second term on the right-hand side of \eqref{eq:sts}.
This gives us
\begin{eqnarray*}
\left| \partial_{y_1}\widetilde{\psi}(y_1,0,..,0) \right| & \le &  \sum_{Z_s \in \gamma_s} \int_{v \in \mathbb{R}^n}\left| Z f\left(x+(t+|x|)(y_1,0,..,0)\right) \right|d^n v.
\end{eqnarray*}
Combined with \eqref{eq:rp1}, we have obtained that 
\begin{eqnarray*}
|\widetilde{\psi}(0)| &\lesssim& \sum_{Z_s \in \gamma_s}\int_{|y_1| \le \delta^{1/2}} \int_{v \in \mathbb{R}^n}\left| Z f\left(x+(t+|x|)(y_1,0,..,0)\right) \right|d^n v dy_1\\
&&\hbox{}+ \int_{|y_1| \le \delta^{1/2}} \int_{v \in \mathbb{R}^n}\left| f\left(x+(t+|x|)(y_1,0,..,0)\right) \right|d^n v dy_1
.
\end{eqnarray*}
We now repeat the same argument varying the variable $y_2$. 

\begin{eqnarray*}
&&|\widetilde{\psi}(0)| \lesssim \int_{|y_1| \le \delta^{1/2}}\int_{|y_2| \le \delta^{1/2}} \\
&&\hbox{}\quad\quad\quad \sum_{\substack{ |\alpha|\le 2,\\ Z^{\alpha} \in \gamma^{|\alpha|}_s} } \int_{v \in \mathbb{R}^n}\left| Z^\alpha f\left(x+(t+|x|)(y_1,y_2,0,..,0)\right) \right|d^n v dy_1 dy_2.
\end{eqnarray*}
Iterating again this argument until all variables $y_i$ appears in the integral of the right-hand side, we obtain
$$
|\widetilde{\psi}(0)| \lesssim  \int_{y \in B_n(0,1/2)} \sum_{\substack{ |\alpha|\le n, \\ Z^{\alpha} \in \gamma^{|\alpha|}_s} } \int_{v \in \mathbb{R}^n}\left| Z^\alpha f\left(x+(t+|x|)y\right) \right|d^n v dy,
$$
and the conclusion of the proof follows as in the proof of Lemma \ref{lem:ksi} by the change of variable $z=(t+|x|)y$.
\end{proof}

We now recall that if $f$ is a solution to $T(f)=0$, then, in view of the commutation properties of Lemma \eqref{lem:cp} and standard properties of differentiation of the absolute value, so are $|f|$ and the commuted fields $|Z^\alpha f|$. Thus, from Lemma \ref{lem:cl}, all the norms on the right-hand side of \eqref{es:gksiva1} are preserved by the flow and we have obtained the decay estimate
\begin{proposition}[Decay estimates for velocity averages]
For any sufficiently regular solution $f$ to $T(f)=0$, we have, for all $t > 0$ and all $x \in \mathbb{R}^n$, 
\begin{equation}\label{es:gksiva}
|\rho(|f|)|(t,x) \lesssim \frac{1}{\left(1+t+|x| \right)^n} \sum_{|\alpha| \le n, Z^\alpha \in \gamma^{|\alpha|}_s} ||  Z^\alpha f(t=0)  ||_{L^1( \mathbb{R}_x \times \mathbb{R}_v )}.
\end{equation}
\end{proposition}

\begin{remark}
If we restrict the set of vector fields only to the uniform motions $t \partial_{x^i} + \partial_{v^i}$, then we still obtain the time decay estimate, 
\begin{eqnarray} \label{eq:tdvf}
|\rho(|f|)|(t,x) \lesssim \frac{1}{t^n} \sum_{\substack{ |\alpha| \le n,\\ Z^{\alpha_i}=t \partial_{x^j} + \partial_{v^j}}} ||Z^\alpha f(t=0)||_{L^1( \mathbb{R}_x \times \mathbb{R}_v )}.
\end{eqnarray}
Now, note that the vector fields of the form $t \partial_{x^j} + \partial_{v^j}$ degenerate to $\partial_{v^j}$ when evalulated at $t=0$.
Since from the Sobolev inequality, we have that
$$|| f(x,.) ||_{L^\infty(\mathbb{R}^n_v)} \lesssim \sum_{|\alpha| \le n} || \partial_{v}^\alpha f(x,.) ||_{L^1(\mathbb{R}_v)},$$
it follows that \eqref{eq:tdvf} is striclty weaker than the inequality \eqref{es:bd}. In some sense, the fact that our method gives us a slightly worse estimate reflects its robustness and thus its appropriateness to deal with non-linear problems.
\end{remark}
\begin{remark}Estimates similar to \eqref{es:gksiva} hold for the relativistic transport operator $T_{m}=\sqrt{m^2+|v|^2} \partial_t + \sum_{i=1}^n v^i \partial_{x^i} $, where $m=0$ for massless particles\footnote{In that case, the decay estimates are worse near the cone $t=|x|$, as for the wave equation, see again \cite{fjs:vfmrgt}.} and $m>0$ for massive particles. These estimates are presented in \cite{fjs:vfmrgt} together with non-linear applications to the Vlasov-Nordstr\"om system. Interestingly, in the case of the relativistic transport operator, the estimates obtained have additional benefits. Decay estimates for relativistic operators in the style of the Bardos-Degond estimate \eqref{es:bd} typically require the extra assumptions of compact support in $v$ of the solution, while for the estimates obtained via the vector field method, only the finiteness of the $L^1_{x,v}$ norms of the commuted fields are required. See \cite{fjs:vfmrgt}. 
\end{remark}

As is classical, derivatives enjoy better decay properties as follows\footnote{Note that however, quantities such as $\rho\left( \left|\partial_{x} f \right| \right)$ do not typically enjoy any additional decay.}
\begin{proposition}[Improved decay estimates for derivatives] \label{prop:idd}
For any sufficiently regular function $f$ defined on $\mathbb{R}^n_x \times \mathbb{R}^n_v$, we have, for all $x \in \mathbb{R}^n$ and all multi-index $\alpha$
$$
| \rho(\partial_{x}^\alpha (f) )|(x) \lesssim \frac{1}{\left(1+t+|x| \right)^{n+| \alpha |}} \sum_{|\beta| \le n+|\alpha|, Z^\beta \in \gamma^{|\beta|}_s} ||  Z^\beta f  ||_{L^1( \mathbb{R}_x \times \mathbb{R}_v )},
$$
Similarly, for any sufficiently regular function $f$ defined on $\mathbb{R}_t \times \mathbb{R}^n_x \times \mathbb{R}^n_v$, we have, for all $t \ge 0$ all $x \in \mathbb{R}^n$ and all $\tau \in \mathbb{N}$,
$$
| \rho(\partial_{x}^\alpha \partial_{t}^\tau (f) )|(t,x) \lesssim \frac{1}{\left(1+t+|x| \right)^{n+|\alpha|} \left(1+t\right)^{\tau}}\sum_{|\beta| \le n+\tau+|\alpha|, Z^\beta \in \gamma^{|\beta|}} ||  Z^\beta f  ||_{L^1( \mathbb{R}_x \times \mathbb{R}_v )},
$$
\end{proposition}
\begin{proof}


The first part of the proposition is a consequence of Lemma \ref{lem:ttks}, Lemma \ref{lem:vac} and the global Klainerman Sobolev inequality \eqref{es:gksiva0}.

The second part of the proposition follows similarly, using that
$$
t \partial_t \rho (f)= \rho (t \partial_t f)= \rho\left[ t \partial_t f -\sum_{i=1}^n v^i \partial_{v^i}f-n f \right]
$$
with $t \partial_t -\sum_{i=1}^n v^i \partial_{v^i}$ being a linear combination of commuting vector fields as explained in Remark \ref{rk:acvf}. 
\end{proof}
\begin{remark}In the above proposition, additional $t$ derivatives yield only additional $t$ decay and no improvement in terms of $|x|$ decay. This improvement can also be achieved assuming stronger decay in $v$ of the initial data. 
More precisely, if $f$ is a solution to $T(f)=0$ then 
$$
\partial_t f= -\sum_{i=1}^n v^i \partial_{x^i} f=\sum_{i=1}^n  \partial_{x^i} (-v^i f).
$$
Note now that if $f$ is a solution, so is $v^i f$, for any $v^i$. Thus, for all $i$, $\int_{v} \partial_{x^i} (-v^i f) dv$ enjoys the additional decay $(t+|x|)$ as stated in Proposition \ref{prop:idd}, provided that the $L^1$ norms of $Z^\alpha (v^i f )$ are finite for $|\alpha| \le n+1$. Iterating the procedure, we obtain that each $\partial_t$ derivatives gives an additional decay of $(t+|x|)$.
\end{remark}

For the applications to the Vlasov-Poisson system of this paper, we will also need an $L^2$ based Klainerman-Sobolev inequality to estimate $\nabla Z^\alpha \phi$ pointwise, for $\phi$ a solution to the Poisson equation \eqref{eq:vp2}. The estimate that we will used is contained in the following lemma.
\begin{lemma}\label{lem:ksmp}
For any sufficiently regular function of $\psi$ defined on $\mathbb{R}^n_x$, we have, for all $t\ge0$, 
$$
|\psi(x)| \lesssim \frac{1}{\left(1+t+|x| \right)^{n/2}} \sum_{|\alpha| \le (n+2)/2, Z^\alpha \in \Gamma_s} || Z^\alpha (\psi) ||_{L^2(\mathbb{R}_x^n)}.$$
\end{lemma}
\begin{proof}The proof is similar to that of Lemma \ref{lem:ksi}, considering the function $\widetilde{\psi}: y \rightarrow \psi(x+(t+|x|)y)$ and replacing the $L^1$ Sobolev inequality with the $L^2$ Sobolev inequality 

\begin{equation} 
|\psi(x)|^2=|\widetilde{\psi}(0)|^2 \lesssim \sum_{|\alpha| \le (n+2)/2} \int_{B_n(0,1/2)}|\partial^\alpha (\widetilde{\psi})|^2 dy. \nonumber
\end{equation}
\end{proof}
\section{Small data solutions of the Vlasov-Poisson system}
\subsection{The norms} \label{se:tn}
The vector fields presented in Section \ref{se:mmvf} and the resulting decay estimates of the previous section will be sufficient to prove global existence of solutions to the Vlasov-Poisson system and derive their asymptotics for all dimension $n \ge 4$. On top of the boundedness of the $L^1$ norms of commuted fields $Z^\alpha(f)$, which are needed in order to obtain pointwise decay from the vectorfield method, we will also need a little bit of additional integrability to prove the $L^p$ boundedness of the gradient of the commuted potentials $\nabla Z^\alpha \phi$.

With this in mind, for any $n \ge 4$, $N \in \mathbb{N}$ and $\delta> 0$, let us consider, for any sufficiently regular function $g$ of $(x,v)$, the norm $E_{N,\delta}$ defined by
$$
E_{N,\delta}[g]:= \sum_{|\alpha|\le N, Z^\alpha \in \gamma_s^{|\alpha|} } || Z^\alpha(g) ||_{L^1(\mathbb{R}^n_x \times \mathbb{R}^n_v)}+\sum_{|\alpha|\le N, Z^\alpha \in \gamma_s^{|\alpha|} } || (1+|v|^2)^{\frac{\delta(\delta+n)}{2(1+\delta)} } Z^\alpha(g) ||_{L^{1+\delta}(\mathbb{R}^n_x \times \mathbb{R}^n_v)}.
$$

As we shall see below, applying a similar strategy would fail in dimension $3$ due to the lack of sufficiently strong decay. In order to close the estimates, it will thus be necessary to improve the commutation relation between our commutation vector fields and our perturbed transport operator. This led us to the introduction of \emph{modified} vector fields, denoted $Y$ (and $Y^\alpha$ for a combination of $|\alpha|$ such vector fields) below. Our main results are then similar to the $n \ge 4$ case, replacing the $Z$ vector fields by the $Y$ ones.
In dimension $3$, the norm $E_{N,\delta}$ will therefore be defined as

\begin{equation}\label{eq:tn}
E_{N,\delta}[g]:= \sum_{|\alpha|\le N, Y^\alpha \in \gamma_{m,s}^{|\alpha|} } || Y^\alpha(g) ||_{L^1(\mathbb{R}^3_x \times \mathbb{R}^3_v)}+\sum_{|\alpha|\le N, Y^\alpha \in \gamma_{m,s}^{|\alpha|} } || (1+|v|^2)^{\frac{\delta(\delta+3)}{2(1+\delta)}} Y^\alpha(g) ||_{L^{1+\delta}(\mathbb{R}^3_x \times \mathbb{R}^3_v)}.
\end{equation}
The precise definitions of the modified vector fields $Y$ and of the algebra $\gamma_{m,s}$ are given in Section \ref{se:dmvfa}.
\begin{remark}
In order to close our main estimates, we will not need to commute with any vector field containing $t$ derivatives, i.e. commuting with vector fields in $\gamma_s$ if $n\ge4$ (respectively $\gamma_{m,s}$ if $n=3$) will be sufficient. Commutations with vector fields containing $t$ derivatives are of course usefull if one wants to obtain decay estimates of $t$-derivatives of $\rho(f)$ and $\nabla \phi$. In that case, one would simply modify the norms, replacing the algebra $\gamma_s$ (respectively $\gamma_{m,s}$) by the algebra $\gamma$ (respectively $\gamma_{m}$). 
The interested reader can then verify that all the arguments below still hold. For these reasons, we will sometimes omit in the following section to specify whether the vector fields considered lie in $\gamma$ (respectively $\gamma_{m}$) or in $\gamma_s$ (respectively $\gamma_{m,s}$).
\end{remark}

\subsection{The main results}\label{se:mr}
Our main results are the following
\begin{theorem}\label{th:mt}Let $n \ge 3$, $0 < \delta < \frac{n-2}{n+2}$, and $N \ge 5n/2+2$ if $n\ge 4$, $N \ge 14$ if $n=3$. Then, 
there exists $\epsilon_0 > 0$ such that for all $0< \epsilon <\epsilon_0$, if 
$E_{N,\delta}[f_0] \le \epsilon$, then the classical solution $f(t,x,v)$ of \eqref{eq:vp1}-\eqref{eq:vp3} exists globally in time and satisfies the estimates
\begin{enumerate}
\item{Global bounds}
$$\forall t \in \mathbb{R},\quad E_{N,\delta}[f(t)] \lesssim \epsilon.$$
\item{Space and time pointwise decay of averages of $f$}
$$\text{for any multi-index $\alpha$ of order $|\alpha| \le N-n$}, \quad 
|\rho( Z^\alpha f)(t,x) | \lesssim \frac{\epsilon}{\left(1+|t|+|x| \right)^{n}},
$$
as well as the improved decay estimates
$$
|\rho( \partial_{x}^\alpha f)(t,x) | \lesssim \frac{\epsilon}{\left(1+|t|+|x| \right)^{n+|\alpha|}}.$$
\item{Boundedness for $L^{1+\delta}$ norms of $\nabla^2 Z^\alpha \phi$}
$$\text{for any multi-index $\alpha$ of order $|\alpha| \le N$ }, \quad ||\nabla^2 Z^\alpha \phi ||_{L^{1+\delta}(\mathbb{R}^n)} \lesssim \epsilon.$$
\item{Space and time decay of the potential and its derivatives}
\begin{align}
\text{for any multi-index $\alpha$ with\,\,}& |\alpha| \le N-(3n/2+1), \nonumber \\
&|Z^\alpha \nabla \phi (t,x) | \lesssim \frac{\epsilon}{t^{(n-2)/2}\left(1+|t|+|x| \right)^{(n)/2}}\nonumber 
\end{align}
as well as the improved decay estimates
$$
|\partial^\alpha_x \nabla \phi (t,x) | \lesssim \frac{\epsilon}{t^{(n-2)/2}\left(1+|t|+|x| \right)^{n/2+|\alpha|}}. 
$$
\end{enumerate}
Finally, all the constants in the above inequalities depend only on $N,n,\delta$.
\end{theorem}
%

As explained above, when $n \ge 4$, the proof is easier and can be performed using only the commuting vector fields of the free transport operator while in the $n=3$ case, we will need to use modified vector fields. The modified vector fields approach would of course work also in the $n \ge 4$ case\footnote{An interesting question left open by our work is to understand what happen in dimension $1$ and $2$ where decay is even sparser than in dimension $3$. We believe that at least in dimension $2$, a carefull analysis using modified vector fields would lead to similar conclusions. We hope to treat this in future work.} but are not necessary there. In order to better explain the general framework, we will treat first the dimension $n \ge 4$ before turning to the proof in the $n=3$ case.
\section{Proof when $n \ge 4$}
In this section, we assume that $n \ge 4$, that the assumptions of Theorem \ref{th:mt} are satisfied for some initial data $f_0$ and we denote by $f$ the classical solution of \eqref{eq:vp1}-\eqref{eq:vp3} arising from $f_0$.

\subsection{Bootstrap assumption}
We will assume the following bootstrap assumption on the norm of the solution $E_{N,\delta}[f(t)]$. Let $T  \ge 0$ be the largest time such that the following bounds hold
\begin{equation} \label{eq:bass}
\forall t \in [0,T],\quad E_{N,\delta}[f(t)] \le 2 \epsilon.
\end{equation}

It follows from the smallness assumptions on $E_{N,\delta}[f_0]$ and a continuity argument that $T> 0$. 

\subsection{Immediate consequence of the bootstrap assumption}
Applying our decay estimate \eqref{es:gksiva} to $\left| Z^\alpha f \right|$ and $\left| Z^\alpha f \right|^p (1+v^2)^{q/2}$ as well as the improved decay estimates of Proposition \ref{prop:idd}, we automatically obtain from our bootstrap assumption \eqref{eq:bass} 
\begin{lemma}For any multi-index $\alpha$ of order $|\alpha| \le N-n$ and for all $t \in [0,T]$,
\begin{eqnarray}\label{eq:perz}
\left|\rho\left( \left| Z^\alpha f \right|\right) (t,x) \right| \le \frac{C_n \epsilon}{\left(1+|t|+|x| \right)^{n}}, \\
\left|\rho\left( \left| Z^\alpha f \right|^{1+\delta} (1+v^2)^{\frac{\delta(\delta+n)}{2}}\right) (t,x) \right| \le \frac{C_n \epsilon^{1+\delta}}{\left(1+|t|+|x| \right)^{n}},\label{eq:perz2}
\end{eqnarray}
for some constant $C_n >0$ depending on $n$, as well as the improved decay estimates,
$$
\left|\rho\left(  \partial^\alpha f\right) (t,x) \right| \le \frac{C_n \epsilon}{\left(1+|t|+|x| \right)^{n+|\alpha|}}.$$
\end{lemma}

\subsubsection{Estimates on $Z^\alpha(\phi)$}

From standard elliptic estimates, we can also bound an $L^p$ norm of $\nabla^2 Z^\alpha \phi$

\begin{lemma}\label{lem:zpes}
For any multi-index $\alpha$ with $|\alpha|\le N$ and for all $t \in[0,T]$, $$||\nabla^2 Z^\alpha \phi(t) ||_{L^{1+\delta}(\mathbb{R}^n)} \le C_{N,n,\delta} \epsilon, $$
where $C_{N,n,\delta} > 0$ is a constant depending on $N$, $n$ and $\delta$.
\end{lemma}
\begin{remark}
There is obviously no difficulty in propagating higher $L^p$ norms of $\,\nabla^2 Z^\alpha \phi$ provided the initial data for $f$ satisfy additional integrability and decay in $v$. We will not need them to close the estimates of our main theorem, which is why we did not assume the initial bounds on these $L^p$ norms.
\end{remark}
\begin{proof}
Let $p=1+\delta$. From the commuted equation for $Z^\alpha \phi$ and the Calder\'on-Zygmund inequality, we have
\begin{eqnarray*}
|| \nabla^2 Z^\alpha \phi ||_{L^p( \mathbb{R}^n)} &\lesssim &|| Z^\alpha ( \rho(f)) ||_{L^p(\mathbb{R}^n)},\\
&\lesssim& \sum_{|\beta| \le |\alpha|}||  \rho \left( \left| Z^\beta (f ) \right|\right) ||_{L^p(\mathbb{R}^n)}
\end{eqnarray*}
using Lemma \ref{lem:vac}.
Thus the bounds on $\nabla^2 Z^\alpha \phi$ follow if we can prove $L^p$ bounds on the $\rho\left( \left| Z^\beta f \right|\right)$. 
For this, let us note that for any weight function $\chi(v)$, we have, using the H\"older inequality with $1/p+1/q=1$
$$
\int_x \left( \int_v  |Z^\beta(f)| dv \right)^p dx \le \int_x \left( \int_v \frac{1}{\chi(v)} dv \right)^{p/q} dv \int_v \chi(v)^{p/q} |Z^\beta(f))|^p dv dx.
$$ 
Thus, we need $\frac{1}{\chi(v)}$ to be integrable in $v$ and we choose $\chi(v)=(1+|v|^2)^{(\delta+n)/2}$. The lemma then follows from the bound on $E_{N,\delta}[f(t)]$ noting that with $p=1+\delta$, we have $p/q=\delta$.

\end{proof}

Applying the Gagliardo-Niremberg inequality, we have immediately
\begin{corollary} \label{cor:lqb}Let $q=\frac{n(1+\delta)}{n-(1+\delta)}$. For all $t\in [0,T]$,
$$
||\nabla Z^\alpha \phi(t) ||_{L^q(\mathbb{R}^n)} \le C_{N,n,\delta} \epsilon. $$ 
\end{corollary}

Applying $L^2$ estimates for solutions to the Poisson equation \eqref{eq:cp} and the previous pointwise estimates on $\rho\left( \left| Z^\alpha(f) \right|\right)$, we can also obtain $L^2$ decay estimates for $\nabla Z^\alpha \phi$ provided $|\alpha|$ is not too large. 
\begin{lemma} \label{lem:l2dp}For all multi-index $\alpha$ such that $|\alpha| \le N-n$
$$
||\nabla Z^\alpha \phi ||_{L^2(\mathbb{R}^n)} \lesssim  \frac{ \epsilon}{t^{(n-2)/2}}.$$
\end{lemma}
\begin{proof}
Multiply the Poisson equation satisfied by $Z^\alpha \phi$ by $Z^\alpha \phi$ and integrate by parts to obtain
\begin{eqnarray*}
||\nabla Z^\alpha( \phi) ||_{L^2(\mathbb{R}^n)}^2&=&-\int_{x \in \mathbb{R}^n}Z^\alpha( \phi)Z^\alpha( \rho(f)) dx \\
&\le&||Z^\alpha( \phi)||_{L^{\frac{2n}{n-2}}(\mathbb{R}^n)}||Z^\alpha( \rho(f))||_{L^{\frac{2n}{n+2}}(\mathbb{R}^n)}.
\end{eqnarray*}
Using the Gagliardo-Nirenberg inequality $||\psi||_{L^{\frac{2n}{n-2}}(\mathbb{R}^n)} \lesssim || \nabla \psi ||_{L^2(\mathbb{R}^n)}$ and Lemma \ref{lem:vac}, we obtain
\begin{eqnarray}
||\nabla Z^\alpha( \phi) ||_{L^2(\mathbb{R}^n ) } &\lesssim&  ||Z^\alpha( \rho(f))||_{L^\frac{2n}{n+2}(\mathbb{R}^n)} \nonumber\\
&\lesssim&  \sum_{|\beta| \le |\alpha|} ||\rho(Z^\beta(f))||_{L^\frac{2n}{n+2}(\mathbb{R}^n)} \label{es:gpl2d}
\end{eqnarray}
Since 
\begin{eqnarray*}
\int_{x \in \mathbb{R}^n } \frac{dx}{\left(1+|x|+t \right)^{\frac{2n^2}{n+2}}} \lesssim t^{-\left(\frac{2n^2}{n+2}-n\right)} \int_{x \in \mathbb{R}^n } \frac{d(x/t)}{\left(|x/t|+1 \right)^{\frac{2n^2}{n+2}}} \lesssim  t^{-\left(\frac{2n^2}{n+2}-n\right)},
\end{eqnarray*}
the lemma now follows by estimating the right-hand side of \eqref{es:gpl2d} using the pointwise estimates \eqref{eq:perz} on $\rho(Z^\beta(f))$.

\end{proof}

The previous lemma combined with the $L^2$ based Klainerman-Sobolev inequality of Lemma \ref{lem:ksmp} gives the pointwise estimates on $\nabla Z^\alpha( \phi)$ claimed in the theorem.
\begin{corollary}\label{lem:dgpp}
For any multi-index $|\alpha| \le N-(3n/2+1)$ and $Z^\alpha \in \Gamma_s^{|\alpha|}$, we have for all $t \in [0,T],$
$$
\left| \nabla Z^\alpha \phi \right| \lesssim \frac{ \epsilon}{(1+|x|+t)^{n/2}t^{(n-2)/2}}.
$$
\end{corollary}


\subsection{Improving the global bounds} \label{se:igb}
\subsubsection{$L^1$ estimates of $Z^\alpha f$} \label{se:il1zaf}
We first consider the $L^1$ estimate on $Z^\alpha(f)$.
From Lemma \ref{lem:cl}, we have for all multi-index $|\alpha| \le N$, and all $t\in [0,T]$,

$$
|| Z^\alpha f(t) ||_{L^1_{x,v}} \le|| Z^\alpha f(0) ||_{L^1_{x,v}} + \int_0^t || T_\phi (Z^\alpha(f) )||_{L^1_{x,v}}.
$$
Thus, we only need to prove that the term below the integral is integrable in $t$.

Now, from the commutation formula of Lemma \ref{lem:comf}, we know that 
$$
|| T_\phi (Z^\alpha(f) )||_{L^1_{x,v}} \le C_N \sum_{|\beta| \le |\alpha|-1,} \sum_{|\gamma|+|\beta| \le |\alpha|}| C_{\beta \gamma}^\alpha| ||\nabla_x Z^\gamma \phi \nabla_v Z^\beta f ||_{L^1_{x,v}} . $$
Note that we have so far no estimates on $v$ derivatives of $Z^\beta f$. To circumvent this difficulty, let us rewrite any $v$ derivative as
$$
\partial_{v^i} Z^\beta(f)= (t \partial_{x^i} + \partial_{v^i} ) Z^\beta(f) -t \partial_{x^i} Z^\beta(f).
$$
Since both $t \partial_{x^i} + \partial_{v^i}$ and $\partial_{x^i}$ are part of our algebra of commuting vector fields, we have 
\begin{equation} \label{eq:tlos}
|| T_\phi (Z^\alpha(f) )||_{L^1_{x,v}} \le C_N (1+t ) \sum_{\substack{ |\beta| \le |\alpha|,|\gamma| \le |\alpha| \\ |\gamma| +|\beta|\le |\alpha|+1}} | C_{\beta \gamma}^\alpha| ||\nabla_x Z^\gamma(\phi)  Z^\beta (f) ||_{L^1_{x,v}} . 
\end{equation}

Now, since $N \ge 5n/2+2$ and $|\gamma| +|\beta|\le |\alpha|+1 \le N+1$, we always have at least either $|\gamma| \le N-(3n/2+1)$ or $|\beta| \le N -n$. 

Case $1$: $|\gamma| \le N-(3n/2+1)$\\
In that case, we have access, thanks to Corollary \ref{lem:dgpp} to the pointwise estimates, for all $t \in [0,T]$,
$$
| \nabla Z^\beta \phi(t,x) | \le \frac{C_N \epsilon}{(1+t)^{n-1}}.
$$
Thus, we get the estimates
\begin{equation} \label{es:dn4}
  ||\nabla_x Z^\gamma(\phi)  Z^\beta (f) ||_{L^1_{x,v}}\le \frac{C_N \epsilon }{(1+t)^{n-2}} \sum_{|\beta| \le |\alpha|} ||Z^\beta f ||_{L^1_{x,v}} ,
\end{equation}
and we see that if $n \ge 4$, then the error is integrable.

Case $2$: $|\beta| \le N-n$\\
In that case, we estimate the error term as follows. First, 
\begin{eqnarray*}
||\nabla_x Z^\gamma(\phi)  Z^\beta (f) ||_{L^1_{x,v}}&=&\int_x \int_v |\nabla_x Z^\gamma(\phi)  | | Z^\beta (f)| dx dv \\
&=&\int_x |\nabla_x Z^\gamma(\phi)  | \rho\left( | Z^\beta (f)| \right)dx,
\end{eqnarray*}
since $Z^\gamma(\phi)$ is independent of $v$.
Now applying the H\"older inequality with $1/p+1/q=1$, we obtain
\begin{eqnarray*}
||\nabla_x Z^\gamma(\phi)  Z^\beta (f) ||_{L^1_{x,v}}&\le& || \nabla_x Z^\gamma(\phi) ||_{L^q_x}  \left|\left| \rho\left( | Z^\beta (f)| \right)\right|\right|_{L^p_x}
\end{eqnarray*}
We would like to take $q=2$, since the $L^2$ bounds are so easy to obtain for solutions to the Poisson equation. Using the pointwise bounds to estimate $|| \rho\left( | Z^\beta (f)| \right)||_{L^2_x}$, 
we would obtain 
$$
\left|\left| \rho\left( | Z^\beta (f)| \right)\right|\right|_{L^2_x} \le C_N \epsilon t^{-n/2}.
$$
On the other hand, we would get no decay a priori on $||\nabla_x Z^\gamma(\phi)||_{L^2_x}$, since $|\gamma|$ is a priori too large to have access to decay estimates for the source term of the Poisson equation satisfied by $Z^\gamma(\phi)$. 
With the extra weight of $t$ in \eqref{eq:tlos}, we see that if $n=4$, we only get $1/t$ decay and we would get a logarithmic loss. 

To avoid this problem, we want to take  $q < 2$. Recalling the $L^q$ bounds of Corollary \ref{cor:lqb}, we see that since $\delta < \frac{n-2}{n+2}$, we have $q=\frac{n(1+\delta)}{n-(1+\delta)} < 2$ and thus, 
$$
|| \rho\left( | Z^\beta (f)| \right)||_{L^p_x} \le\frac{C_N \epsilon}{(1+ t)^{n/2-1+\sigma}},
$$
for some $\sigma > 0$, which is now integrable in $t$. Putting everything together, we have obtain
$$
|| T_\phi (Z^\alpha(f) )||_{L^1_{x,v}} \le\frac{C_N \epsilon}{(1+ t)^{n-2}} E_{N,\delta}[f(t)]+\frac{C_N \epsilon^2}{(1+ t)^{n/2-1+\sigma}} \lesssim \frac{C_N \epsilon^2}{(1+ t)^{n/2-1+\sigma}},
$$
using the bootstrap assumptions to bound $E_{N,\delta}[f(t)]$. 
\subsubsection{$v$ weighted $L^{p}$ estimates of $Z^\alpha f$}
Let $p=1+\delta$ and $q=\delta(\delta+n)$. Recall from \eqref{eq:Ewv} the inequality,
\begin{eqnarray*}
\left|\left|(1+v^2)^{q/2} |Z^\alpha f(t)|^p \right|\right|_{L^1( \mathbb{R}^n_x \times \mathbb{R}^n_v)}&\lesssim& \left|\left| (1+v^2)^{q/2} |Z^\alpha f(0)|^p \right|\right|_{L^1( \mathbb{R}^n_x \times \mathbb{R}^n_v)} \\
&&\hbox{}+\int_0^t \left|\left| (1+v^2)^{q/2}  |Z^\alpha f |^{p-1} T_\phi(Z^\alpha f) \right|\right|_{L^1( \mathbb{R}^n_x \times \mathbb{R}^n_v )}ds \\
&&\hbox{}+\int_0^t \left|\left|\,\, | Z^\alpha  f |^{p}  (1+v^2)^{(q-1)/2} \partial \phi \right|\right|_{L^1( \mathbb{R}^n_x \times \mathbb{R}^n_v )}ds, \\
&\lesssim& \left|\left| (1+v^2)^{q/2} |Z^\alpha f(0)|^p \right|\right|_{L^1( \mathbb{R}^n_x \times \mathbb{R}^n_v)}+I_1+I_2,
\end{eqnarray*}
where $$I_1=\int_0^t \left|\left| (1+v^2)^{q/2}  |Z^\alpha f |^{p-1} T_\phi(Z^\alpha f) \right|\right|_{L^1( \mathbb{R}^n_x \times \mathbb{R}^n_v )}$$
and
$$
I_2=\int_0^t \left|\left|\, \,| Z^\alpha  f |^{p}  (1+v^2)^{(q-1)/2} \partial \phi \right|\right|_{L^1( \mathbb{R}^n_x \times \mathbb{R}^n_v )}.
$$

To estimate the error term $I_1$, we proceed as in the previous section, replacing $T_\phi(Z^\alpha f)$ using the commutation formula of Lemma \ref{lem:comf} and rewriting the terms of the form $\partial_v^i Z^\beta f$ as $\left( t\partial_{x^i}+ \partial_v \right) Z^\beta(f)- t \partial_{x^i} Z^\beta f$. We are then left with error terms of the form
\begin{equation}\label{err:i1t}
(1+t)\int_{x} \int_{v} |\partial_x Z^\gamma \phi| |Z^\beta f|(1+v^2)^{q/2} |Z^\alpha f|^{p-1} dxdv,
\end{equation}
which need to be integrable in $t$. Using Young inequality, we have that 
$$
\int_v |Z^\beta f|(1+v^2)^{q/2} |Z^\alpha f|^{p-1} dv \lesssim \int_v |Z^\beta f|^p(1+v^2)^{q/2} dv + \int_v |Z^\alpha f|^p(1+v^2)^{q/2} dv.$$
Note moreover that, as in the previous section, if $|\alpha| \le N-n$, we have access to the pointwise estimates \eqref{eq:perz2}. With this in mind, all the error terms coming from $I_1$ can then be estimated as in the previous section. 


The estimates on the error term $I_2$ are easier and rely on the pointwise estimates of $\partial \phi$ of Lemma \ref{lem:dgpp}.

\subsection{Conclusions of the proof of the $n \ge 4$ case}
Thus, we have obtained
$$
E_{N,\delta}(t) \le E_{N,\delta}(0)+ C_N \int_0^t \frac{\epsilon^2}{(1+ s)^{n/2+\sigma}}  ds,
$$
for some $\sigma > 0$. 
Using the smallness assumption $E_{N,\delta}(0) \le \epsilon$, it follows that $E_{N,\delta}(t) \le \epsilon + C \epsilon^2 \le \frac{3}{2} \epsilon$, provided $\epsilon$ is sufficiently small, which improves our original bootstrap assumption \eqref{eq:bass} and concludes the proof.

\section{The three dimensional case}
\subsection{Strategy of the proof}
Repeating the previous argument in the $n=3$ case would fail as seen from inequality \eqref{es:dn4}. Moreover, the non-linear terms do not seem to possess any special structure that would allow better decay such as the null condition for non-linear waves in $3d$, and it does not seem possible to close the estimates allowing the norms to grow slowly, using some hierarchy in the equations, as it sometimes happen for some system of non-linear evolution equations (for instance in \cite{lr:gsmhg}). Thus, it seems that one is forced to try to improve the commutation relations so as to remove the most problematic error terms. This will be done using modified vector fields. In hindsight, this strategy is reminiscent of the strategy of \cite{ck:gnsms} where modified vector fields are constructed by solving transport equations along null cones. 

To understand and motivate the definitions of the modified vector fields that we will use here, let us consider a solution $f$ to the transport equation $T_\phi(f)=0$ and commute this equation with $Z_i=t \partial_{x^i} + \partial_{v^{i}}$. We obtain

\begin{eqnarray*}
T_\phi ( Z_i(f))&=& T( Z_i(f)) + \mu \nabla_x\phi. \nabla_v Z_i(f) \\
&=& Z_i(T(f)) + \mu Z_i (\nabla_x \phi.\nabla_v f) -  \mu \nabla_x (Z_i \phi ). \nabla_v f \\
&=&Z_i ( T_\phi (f)) -\mu\sum_{j=1}^n \partial_{x^j} Z_i( \phi). ( t\partial_{x^j}+\partial_{v^j} )f + \mu \sum_{j=1}^n \partial_{x^j} Z_i( \phi). t \partial_{x^j} f
\end{eqnarray*}
Thus, the error terms on the right-hand side are of two forms. 
The good terms are of the form $\partial_x Z_i(\phi) Z'f$, where the $Z'$ are some of the commuting vector fields.  These have enough decay so that they can be estimated as before. The bad terms are of the form $t\partial_x Z_i (\phi) Z'f$ where, in view of the extra $t$ factor, the previous arguments would lead to logarithmic growth. 


The aim of the modified vector fields will be to avoid the introduction of such bad terms. Note that on the other hand, commutations with vector fields such as $\partial_t$ or $\partial_{x^i}$ would be better, because $\partial \partial \phi$ enjoys improved decay. As a consequence, we will only need to modify the homogeneous vector fields. 

Let us consider, for all $1 \le i \le n$, vector fields of the form
$$Y_i=t \partial_{x^i} + \partial_{v^i} - \sum_{j=1}^n \Phi^{j}_i(t,x,v) \partial_{x^j},$$
where the coefficients $\Phi^{j}_i(t,x,v)$ are sufficiently regular functions to be specified below. 


Since $\partial_{x^j}$ commute with the free transport operator, we have the commutation formula

\begin{eqnarray}
[T_\phi, Y_i](f)&=& -\mu\sum_{j=1}^n \partial_{x^j} Z_i(\phi) Z_j(f)
 +\mu \sum_{j=1}^n \Phi^{j}_i \partial_{x^j}(\nabla_x\phi).\nabla_{v}f \label{eq:mcf}\\
&&\hbox{}-\sum_{j=1}^n T_\phi ( \Phi^{j}_i) \partial_{x^j}f
+ \mu\sum_{j=1}^n \partial_{x^j} Z_i ( \phi) t \partial_{x^j} f \nonumber
\end{eqnarray}
where $Z_j = t \partial_{x^j} + \partial_{v^j}$.
The first term on the right-hand can be handled as before as it does not have $v$ derivatives leading to the extra power of $t$. 
Note moreover that since $t \partial_{x^j }$ is part of the algebra of macroscopic commuting vector field, the second term can be rewritten as 
$$
\sum_{j=1}^n \Phi^{j}_i t^{-1} \nabla_x Z_j \phi \nabla_v f
$$
and thus is expected to be integrable by the previous arguments provided $\Phi^{j}_i$ are uniformly bounded by any power of $t$ strictly less $1$.


The key idea is then to choose appropriately the $\Phi^{j}_i$ so as to be able to cancel the last two terms in \eqref{eq:mcf}. For this, we impose that each  $\Phi^{j}_i$ is obtained as the unique solution to the inhomogeneous transport equation
$$
T_\phi (\Phi^{j}_i )= \mu t \partial_{x^j} Z_i(\phi) 
$$
with $0$ initial data. Note that the right-hand side of this equation only decay like $1/t$, so we expect $\Phi^{j}_i $ to actually grow logarithmically in $t$. 

Assuming that this holds, we see that all the error terms in \eqref{eq:mcf} are now integrable and thus, we should expect to control our norms $E_{N,\delta}[f(t)]$ provided all the original commutation vector fields $Z$ are replaced by their modified form $Y$, hence the definition of the norm in the $3$d case given by \eqref{eq:tn}. A few difficulties remain
\begin{enumerate}
\item Since the $\Phi^j_i$ (and in fact all the coefficients involved in the contruction of the modified vector fields) are growing logarithmically, it is apriori not clear how to exploit the new energies obtained after commutation with modified vector fields to obtain pointwise decay. The idea is again to rewrite the extra terms such as $\Phi^{j}_i \partial_{x^j}(f)$ in the form
$$
 \Phi^{j}_i \partial_{x^j}(f)= \frac{\Phi^j_i}{t} \left( t \partial_{x^j} + \partial_{v^j} \right)(f)-\frac{\Phi^j_i}{t} \partial_{v^j}(f)
$$
and to use an integration by parts in $v$ for the last term to push the $v$ derivatives on the $\Phi^j_i$ coefficient. See Section \ref{se:ksimvf} below. 
\item The Poisson equation $\Delta \phi=\rho(f)$ cannot be commuted with modified vector fields, since the coefficients in the modified vector fields depend on $v$, while $\phi$ and $\rho(f)$ are macroscopic quantities depending only on $(t,x)$. Thus, we keep commuting this equation with non-modified vector fields. Quantities such as $Z^\alpha \rho(f)$ are then rewritten as $\rho( Y^\alpha(f))$ plus error terms. The structure of these error terms is the subject of Lemma \ref{lem:pcm}. 
\item As is seen from the statement of Lemma \ref{lem:pcm}, some of the error terms will depend on $Y^\alpha(\varphi)$, where $\varphi$ is a coefficient obtained by solving a transport equation of the form $T_\phi(\varphi)= t \partial_x Z \phi$ and $Y^\alpha$ is a composition of $|\alpha|$ modified vector fields. Commuting the transport equation satisfied by $\varphi$ by $|\alpha|$ vector fields, we see that we need to control $t Y^\alpha \partial_x Z (\phi)$. This poses a problem at the top order, when $|\alpha|=N$, since it looks as if one needs to control $\partial Z^\alpha Z(\phi)$, which would a priori require to commute the Poisson equation by $N+1$ vector fields, and therefore would forbid us to close the estimates. On the other hand, using directly the Poisson equation satisfied by $Z^\alpha(\phi)$, we may hope to control $\partial^2 Z^\alpha (\phi)$. To exploit this fact and close the top order estimates, we devote some time to describe the structure of the top order terms of Lemma \ref{lem:pcm} in Lemma \ref{lem:pcmto}.
\item Finally, numerous terms of the form $Y^\alpha(\varphi) Y^\beta(f)$, where $\varphi$ is a coefficient obtained by solving a transport equation as above, will appear in the equations. This creates another difficulty to close the top order estimates, since we do not have access to pointwise estimates on $Y^\alpha(\varphi)$ when $|\alpha|$ is large, nor do we have access to $L^p_{x,v}$ estimates on $Y^\alpha(\varphi)$ because the source terms in their transport equations are not integrable in $v$. We will circumvent this difficulty by considering directly a transport equation satisfied by the product  $Y^\alpha(\varphi) Y^\beta(f)$. See Section \ref{se:ep} below. 
\end{enumerate}
We now turn to the details of the proof. 
\subsection{Definitions and properties of the modified vector fields algebra} \label{se:dmvfa}
As explained above, our new algebra of microscopic vector fields will consist in 
\begin{enumerate}
\item Standard translations $\partial_t$, $\partial_{x^i}$,
\item Modified uniform motions $Y_i:=t \partial_{x^i}+ \partial_{v^i}-\sum_{k=1}^n \Phi^{k}_i \partial_{x^k}$,
\item Modified rotations $\Omega_{i,j}^{x}+\Omega_{i,j}^{v}-\sum_{k=1}^n \omega^{k}_{i,j} \partial_{x^k}$,
\item Modified scaling in space $S^{x}+S^{v}-\sum_{k=1}^n \sigma^{k} \partial_{x^k}$,
\item Modified scaling in space and time $t\partial_t + \sum_{i=1}^n x^i \partial_{x^i} - \sum_{k=1}^n \theta^{k} \partial_{x^k}$,
\end{enumerate}
where the coefficients $\Phi^{k}_i$, $\omega^{k}_{i,j}$, $\sigma^{k}$, $\theta^{k}$ are solutions of the following inhomogeneous equations with $0$ data at $t=0$
\begin{eqnarray*}
T_\phi( \Phi^{k}_i )&=& \mu t \partial_{x^k}\left[ Z_i(\phi) \right], \\
T_\phi( \omega^{k}_{i,j})&=&\mu t \partial_{x^k} \left[\Omega_{i,j}^{x} ( \phi)\right], \\
T_\phi( \sigma^{k} )&=&\mu t \partial_{x^k} \left[S^{x} (\phi)-2\phi\right], \\
T_\phi(  \theta^{k}) &=& \mu t \partial_{x^k}\left[ \left( t\partial_t + \sum_{i=1}^n x^i \partial_{x^i} \right)( \phi)\right].
\end{eqnarray*}
\subsubsection{Further notations}
Since the sign of $\mu$ will play no role in the analysis to come, we will asume without loss of generality that $\mu=1$ in the rest of this article, to symplify the notation. 

We will denote by $\mathcal{M}$ the set of all the coefficients $\Phi^{k}_i$, $\omega^{k}_{i,j}$, $\sigma^{k}$, $\theta^{k}$ and by $\varphi$ a generic coefficient among them. Similarly to the set $\gamma$ and $\gamma_s$, we define the sets $\gamma_{m}$ and $\gamma_{m,s}$ where $\gamma_{m}$ is the set of all modified vector fields (including the translations) and $\gamma_{m,s}$ is the set of all modified vector fields minus the time translation and the space-time scaling.

If $Z$ is an original, non-modified vector field, we will sometimes write schematically $Y=Z+\varphi \partial_x$ to denote the associated modified vector field, where by convention $\varphi \partial_x=0$ is $Z$ is any of the translations, $\varphi \partial_x=\displaystyle{\sum_{k=1}^n \Phi^{k}_i \partial_{x^k}}$ if $Z$ is one of the uniform motions and similarly for the other vector fields.

Finally, we will use the notation $P(\bar{\varphi})$ to denote a function depending on all the coefficients in $\gamma_m$ or $\gamma_{m,s}$.

\subsubsection{Improved commutation formulae} \label{se:icf}
With these definitions, we now have the following improved commutation formula\footnote{Once again, in all the formulae of this section, we assume that $\phi$ is a sufficiently regular function of $(t,x)$ defined on $[0,T] \times \mathbb{R}^3$ and we will eventually control the regularity of $\phi$ through a bootstrap argument.}.

\begin{lemma}\label{lem:icf}
For any $Y \in \gamma_{m}$ and any sufficiently regular function $g$ of $(t,x,v)$, $[T_\phi,Y](g)$ can be written as a linear combination with constant coefficients of terms of the form $\partial_{x} Z(\phi) Y(g)$, $\varphi \partial_{x} Z(\phi) Y(g)$ or $\varphi^2 \partial_{x} Z(\phi) Y(g)$ where $\varphi \in \mathcal{M}$, $\varphi^2$ denotes a generic product of two coefficients in $\mathcal{M}$, $Z \in \Gamma$ and $Y \in \gamma_{m}$.
\end{lemma}
\begin{proof}
If $Y$ is a translation, for instance $Y=\partial_{x^k}$, then
\begin{eqnarray*}
[T_\phi, Y](g)&=&-(\partial_{x^k} \nabla_x \phi).\nabla_v g, \\
&=&-\sum_{i=1}^n \partial_{x^i} \partial_{x^k} \phi. \left( t \partial_{x^i}g+\partial_{v^i}g-\sum_{j=1}^n \Phi^{j}_i \partial_{x^j} g \right) \\
&&\hbox{}+\sum_{i=1}^n \partial_{x^i} \partial_{x^k} \phi. \left( t \partial_{x^i}g-\sum_{j=1}^n \Phi^{j}_i \partial_{x^j} g \right) \\
&=&-\sum_{i=1}^n \partial_{x^i} \partial_{x^k} \phi. Y_i(g) \\
&&\hbox{}+\sum_{i=1}^n \partial_{x^i} t \partial_{x^k} \phi. \partial_{x^i}g-\sum_{i=1}^n\sum_{j=1}^n \Phi^{j}_i \partial_{x^i} \partial_{x^k} (\phi). \partial_{x^j} g,
\end{eqnarray*}
which is of the desired form since $t \partial_{x^k} \in \Gamma$. 
If $Y=Y_i=t \partial_{x^i}+\partial_{v^i}-\sum_{j=1}^n \Phi^{j}_i \partial_{x^j}$, then it follows from \eqref{eq:mcf} and the definition of the coefficients $\Phi^{j}_i$ that
\begin{eqnarray} \label{eq:mcf2}
[T_\phi, Y_i](g)=  -\sum_{j=1}^n \partial_{x^j} Z_i(\phi) Z_j(g)
 +\sum_{j=1}^n \Phi^{j}_i \partial_{x^j}(\nabla_x\phi).\nabla_{v}g.
\end{eqnarray}
where $Z_i(\phi)= t \partial_{x^i} \phi $ and $Z_j(g)=t \partial_{x^i}g+\partial_{v^j} g$.
Since $Z_j(g)=Y_j(g)+\sum_{i=1}^n  \Phi^i_j\partial_{x^i}(g)$, the first term term on the right-hand side of \eqref{eq:mcf2} is of the desired form while for the last one, we have
\begin{eqnarray*}
\sum_{j=1}^n \Phi^{j}_i \partial_{x^j}(\nabla_x\phi).\nabla_{v}g&=&
\sum_{j=1}^n\sum_{k=1}^n \Phi^{j}_i \partial_{x^j}\partial_{x^k}\phi.\partial_{v^k}g \\
&=&\sum_{j=1}^n\sum_{k=1}^n \Phi^{j}_i \partial_{x^j}\partial_{x^k}\phi.Y_k(g) \\
&&\hbox{}-\sum_{j=1}^n\sum_{k=1}^n \Phi^{j}_i \partial_{x^j}\partial_{x^k}\phi.\left( t \partial_{x^k}g-\sum_{l=1}^n \Phi^{l}_i \partial_{x^l} g \right),\\
&=&\sum_{j=1}^n\sum_{k=1}^n \Phi^{j}_i \partial_{x^j}\partial_{x^k}\phi.Y_kg \\
&&\hbox{}-\sum_{j=1}^n\sum_{k=1}^n \Phi^{j}_i \partial_{x^j}(t\partial_{x^k}\phi). \partial_{x^k}g
+\sum_{j=1}^n\sum_{k=1}^n\sum_{l=1}^n \Phi^{l}_i \Phi^{j}_i \partial_{x^j}(\partial_{x^k}\phi) \partial_{x^l} g,
\end{eqnarray*}
which is of the desired form. The commutation with the other modified vector fields can be computed similarly. Note that in the case of the spatial vector field, we have $[T_\phi, S^x+S^v]=-  \nabla_x S^x(\phi).\nabla_v + 2 \nabla_x \phi. \nabla_v $, which explains the extra term in the definition of the coefficients $\sigma^k$ compared to the other vector fields.
\end{proof}

If we apply several times the previous formula in order to compute $[T_\phi,Y^\alpha ]$, many products of the form $Y^\rho(\varphi) Y^{\nu}(\varphi')$ will appear. In order to simplify the presentation below, it will be usefull to use the following definition. 

\begin{definition} 
We will say that $P(\bar{\varphi})$ is a multilinear form of degree $d$ and signature less than $k$ if $P(\bar{\varphi})$ is of the form
$$
P(\bar{\varphi})=\sum_{ \substack{\rho \in I^d, \\ \rho=(\rho_1, .., \rho_d)}} C_\rho \prod_{\substack{j=1,..,n,\\ \varphi \in \mathcal{M}}} Y^{\rho_j}(\varphi),
$$
where $I$ denotes the set of all multi-indices (thus $\rho_j$ is a multi-index for each $j$), for each $\rho$ in the above formula $\displaystyle \sum_{i=1}^d |\rho_j| \le k$ and where the $C_\rho$ are constants.
\end{definition}

From Lemma \ref{lem:icf}, we now obtain
\begin{lemma}\label{lem:gcfm}
For any multi-index $\alpha$, we have
\begin{equation} \label{eq:gcfm}
[T_\phi,Y^\alpha ]= \sum_{d=0}^{|\alpha|+1}\sum_{i=1}^n \sum_{\substack{|\gamma| \le |\alpha|,\\ |\beta| \le |\alpha|}} P^{\alpha,i}_{d \gamma \beta}(\bar{\varphi}) \partial_{x^i} Z^\gamma(\phi) Y^\beta,
\end{equation}
where the $P^{\alpha,i}_{d \gamma \beta}(\bar{\varphi})$ are multilinear forms of degree $d$ and signature less than $k$ such that $k \le |\alpha|-1$ and $k+|\gamma|+|\beta| \le |\alpha|+1$. 

\end{lemma}
\begin{proof}This is a classical proof by induction for which we will just sketch the details.
Lemma \ref{lem:icf} shows that \eqref{eq:gcfm} holds when $|\alpha|=1$. Assume that \eqref{eq:gcfm} holds for some multi-index $\alpha$ and let $Y \in \gamma$ be an arbitrary modified vector field. We have
$$
[T_\phi, Y Y^\alpha]=[T_\phi, Y]Y^\alpha+Y[T_\phi, Y^\alpha].
$$
One easily see that the first term on the right-hand side has the correct form using Lemma \ref{lem:icf}. The second term will generate three types of terms. The terms of the form
$$
Y\left(P^{\alpha,i}_{d \gamma \beta}\right) \partial_{x^i} \left(Z^\gamma(\phi) \right)Y^\beta
$$
are of the correct form, the multilinear form being of the same degree and its signature being increased by $1$ at most. 
For the terms of the form 
$$
P^{\alpha,i}_{d\gamma \beta}Y\left( \partial_{x^i}( Z^\gamma(\phi))\right) Y^\beta, 
$$
we recall that $Y$ is schematically of the form $Z+\varphi \partial_x$. Thus, 
$$
P^{\alpha,i}_{d\gamma \beta}Y( \partial_{x^i} Z^\gamma(\phi)) Y^\beta= \sum_{|\gamma'| \le |\gamma|+1}P'_{\gamma'}\partial_{x^i} Z^{\gamma'}(\phi)) Y^\beta,
$$
where $P'_{\gamma'}$ are multilinear forms of degree at most $d+1$ and we have not changed the signature, so that these terms are of the desired form. 


Finally, the last terms are of the form $P^{\alpha,i}_{d\gamma \beta} \partial_{x^i} \left( Z^\gamma(\phi) \right) Y Y^\beta$, which clearly satisfied the required properties.
\end{proof}

As explained above, we also need to revisit our commutation relations for the Poisson equation satisfied by the potential $\phi$. Contrary to $f$, we cannot commute with a modified vector field, because the coefficients of the modified vector fields depend on $v$, while $\phi$ is a macroscopic quantity and depends only on $(t,x)$. Thus, we keep commuting with $Z^\alpha$. We have

\begin{lemma} \label{lem:pcm}
Let $g$ be a sufficiently regular function of $(t,x,v)$ and let $\phi_g$ solves 
$$
\Delta \phi_g =\rho(g).
$$
For any multi-index $\alpha$ with $|\alpha| \le N$, we have, for all $0 < t \le T$,
\begin{equation} \label{eq:crz}
\rho(Z^\alpha(g))= \sum_{j=1}^{|\alpha|} \sum_{d=1}^{|\alpha|+1} \sum_{|\beta| \le |\alpha | } \frac{1}{t^j} \rho\left( P^{\alpha,j}_{d\beta}(\bar{\varphi}) Y^\beta(g) \right)+ \sum_{d=0}^{|\alpha|}\sum_{|\beta| \le |\alpha | } \rho\left( Q^{\alpha}_{d\beta}(\partial_x \bar{\varphi}) Y^\beta(g) \right)
\end{equation}
and 
\begin{equation} \label{eq:cdz}
 \Delta Z^\alpha(\phi) =  \sum_{j=1}^{|\alpha|} \sum_{d=1}^{|\alpha|+1} \sum_{|\beta| \le |\alpha | } \frac{1}{t^j} \rho\left( \tilde{P}^{\alpha,j}_{d\beta}(\bar{\varphi}) Y^\beta(g) \right)+ \sum_{d=0}^{|\alpha|}\sum_{|\beta| \le |\alpha | } \rho\left( \tilde{Q}^{\alpha}_{d\beta}(\partial_x \bar{\varphi}) Y^\beta(g) \right),
\end{equation}
where $P^{\alpha,j}_{d\beta}(\bar{\varphi})$ and $\tilde{P}^{\alpha,j}_{d\beta}(\bar{\varphi})$ are multilinear forms of degree $d$ and signature less than $k$ satisfying 
$$k \le |\alpha|, \quad k+|\beta| \le |\alpha|$$
 and where the $Q^{\alpha}_{d\beta}(\partial_x\bar{\varphi})$ and $\tilde{Q}^{\alpha}_{d\beta}(\partial_x\bar{ \varphi})$ are all multilinear forms of degree $d$ of the form 
\begin{equation} \label{eq:qmf}
Q(\partial_x\bar{\varphi})=\sum_{ \substack{\rho \in I^d, \\\rho=(\rho_1, .., \rho_d)}} C_\rho \prod_{\substack{j=1,..,n,\\ \varphi \in \mathcal{M}}} Y^{\rho_j}(\partial_{x^{r_j}} \varphi),
\end{equation}
where the $C_\rho$ are constants, $1 \le r_j \le 3$, and such that $\displaystyle k':=\sum_{j=1}^d |\rho_j|$ satisfies

$$k' \le |\alpha|-1, \quad d+k'+|\beta| \le |\alpha|.$$ 
\end{lemma}
\begin{proof}
First, recall that if $Z \in \Gamma$, then
$$\Delta Z(\phi_g)= Z \Delta \phi_g + d_Z \Delta \phi_g,$$ where $d_Z=0$ unless $Z$ is one of the two scaling vector fields, in which case $d_Z=2$, and

$$
Z( \rho(g))= \rho(Z(g))+c_Z \rho(g),
$$
where $c_Z=0$ unless $Z$ is the spatial scaling vector field, in which case $c_Z=3$. 
Since $\Delta \phi_g=\rho(g)$, it follows that \eqref{eq:crz} implies \eqref{eq:cdz}.

In the next lines of computations, given $Z \in \gamma$ and $Y$ the modified vector field corresponding to $Z$, we will use the schematic notations $Y=Z-\varphi \partial_x$ and $Z=Y+\varphi \partial_x$ instead of any of the lengthy formulae given at the beginning of Section \ref{se:dmvfa}, such as $Y_i=t \partial_{x^i}+ \partial_{v^i}-\sum_{k=1}^n \Phi^{k}_i \partial_{x^k}$. We will also use the notation $t \partial_{x}+ \partial_{v}-\varphi \partial_x$ to denote a generic vector field among the $Y_i$, the letter $Y'$ to denote a generic modified vector field and the letter $\varphi'$ to denote a generic coefficient belonging to $\mathcal{M}$.

We now compute, for any $Z \in \gamma$
\begin{eqnarray*}
\int_v Z(g)dv&=&\int_v \left( Z+\varphi \partial_x-\varphi \partial_x \right)(g) dv \\
&=& \int_v Y(g)dv - \int_v \varphi \partial_x g dv \\
&=& \int_v Y(g)dv - \int_v \frac{\varphi}{t} \left( t \partial_x g+ \partial_v g-\varphi \partial_x g -\partial_v g +\varphi \partial_xg \right) dv\\
&=& \int_v Y(g) dv -  \frac{1}{t} \int_v \varphi \left( Y'(g)+\varphi \partial_xg \right) dv + \int_v \frac{\varphi}{t} \partial_v g d v
\end{eqnarray*}
The first and second terms on the right-hand side of the last line have the correct forms. For the last term, we integrate by parts in $v$
\begin{eqnarray*}
-\int_v \frac{\varphi}{t} \partial_v g d v&=&\frac{1}{t} \int_v \partial_v \varphi g d v  \\
&=& \frac{1}{t} \int_v  \left( t \partial_x+ \partial_v - \varphi' \partial_x-t \partial_x + \varphi' \partial_x\right)(\varphi) g d v  \\
&=& \frac{1}{t} \int_v  \left(Y'+ \varphi' \partial_x\right)(\varphi) g d v -\int_v \partial_x (\varphi) g dv,
\end{eqnarray*}
where now all terms are of the correct forms. This prove \eqref{eq:crz} when $|\alpha|=1$. 
We now assume that $\eqref{eq:crz}$ is true for some $\alpha$. Let $Z$ be a non-modified vector field. Using again that $Z \rho ( Z^\alpha(g))= \rho( ZZ^\alpha(g)) + c_Z \rho ( Z^\alpha(g))$, we only need to prove that $Z \rho ( Z^\alpha(g))$ is of the correct form. Assume first that $Z$ contains no $t$-derivative\footnote{Recall also that commuting with the vector fields containing $t$-derivatives is not necessary to prove Theorem \ref{th:mt} and is only usefull if one wants to obtain improved decay of $t$-derivatives.}, i.e.~$Z\neq \partial_t$ and $\displaystyle Z\neq t\partial_t +\sum_{i=1}^n x^i \partial_{x^i}$. Using the induction hypothesis and writing $Y=Z+\varphi \partial_x$ to denote the associated modified vector field, we have

\begin{eqnarray}
Z \rho ( Z^\alpha(g))&=&Z \left( \sum_{j=1}^{|\alpha|} \sum_{d=1}^{|\alpha|+1} \sum_{|\beta| \le |\alpha | } \frac{1}{t^j} \rho\left( P^{\alpha,j}_{d\beta}(\bar{\varphi}) Y^\beta(g) \right)+ \sum_{d=1}^{|\alpha|}\sum_{|\beta| \le |\alpha | } \rho\left( Q^{\alpha}_{d\beta}(\partial_x \bar{\varphi}) Y^\beta(g) \right) \right) \nonumber \\
&=& \sum_{j=1}^{|\alpha|} \sum_{d=1}^{|\alpha|+1} \sum_{|\beta| \le |\alpha | } \frac{1}{t^j} \rho\left(Z \left[ P^{\alpha,j}_{d\beta}(\bar{\varphi}) Y^\beta(g)\right] \right)+ \sum_{d=1}^{|\alpha|}\sum_{|\beta| \le |\alpha | } \rho\left(Z \left[Q^{\alpha}_{d\beta}(\partial_x \bar{\varphi}) Y^\beta(g)\right] \right)\nonumber\\
&&\hbox{} + c_Z \rho ( Z^\alpha(g)). \label{inteq:c}
\end{eqnarray}
If now $Z$ contains $t$-derivatives, we would get extra terms in \eqref{inteq:c} which arise when $\partial_t$ hits the $\frac{1}{t^j}$ factors. Since these extra terms are all of the correct forms, so we only need to analyse the terms on the right-hand side of \eqref{inteq:c}.

The last term in \eqref{inteq:c} has already the right form. Next, replacing $Z$ by $Y+\varphi \partial_x$ in the terms $\rho\left(Z \left[ P^{\alpha,j}_{d\beta}(\bar{\varphi}) Y^\beta(g) \right] \right)$, one easily see that they also have the desired form. For the terms $\rho\left(Z\left[ Q^{\alpha}_{d\beta}(\partial_x \bar{\varphi}) Y^\beta(g)\right] \right)$, we have
\begin{eqnarray*}
\rho\left(Z \left[Q^{\alpha}_{d\beta}(\partial_x \bar{\varphi}) Y^\beta(g)\right] \right)&=&\rho\left((Y+\varphi \partial_x ) \left[ Q^{\alpha}_{d\beta}(\partial_x \bar{\varphi}) Y^\beta(g) \right]\right) \\
&=& \rho\left(Y \left[ Q^{\alpha}_{d\beta}(\partial_x \bar{\varphi}) Y^\beta(g) \right]\right) + \rho\left(\varphi \partial_x  \left[ Q^{\alpha}_{d\beta}(\partial_x \bar{\varphi}) Y^\beta(g) \right]\right)
\end{eqnarray*}
The first term on the right-hand side is easily seen to have the correct form. For the second term, we write $\varphi \partial_x=\frac{\varphi}{t} \left( t \partial_x + \partial_v-\varphi' \partial_x+\varphi' \partial_x - \partial_v \right)= \frac{\varphi}{t} Y'- \frac{\varphi}{t} \partial_v + \frac{\varphi \varphi'}{t} \partial_x$ so that
\begin{eqnarray*}
\rho\left(\varphi \partial_x  \left[ Q^{\alpha}_{d\beta}(\partial_x \bar{\varphi}) Y^\beta(g) \right]\right)&=&
\frac{1}{t}\rho\left( \left(\varphi Y'+\varphi \varphi' \partial_x \right) \left[ Q^{\alpha}_{d\beta}(\partial_x \bar{\varphi}) Y^\beta(g) \right]\right)\nonumber \\
&&\hbox{}+\frac{1}{t} \rho\left(\partial_v \varphi \left[ Q^{\alpha}_{d\beta}(\partial_x \bar{\varphi}) Y^\beta(g) \right]\right)
\end{eqnarray*}
using an integration by parts in $v$. The first term on the right-hand side has now the right-form.
For the second term, we again write $\partial_v \varphi= \left(t \partial_x + \partial_v -\varphi' \partial_x \right) \varphi- t \partial_x \varphi+ \varphi' \partial_x \phi$ so that 
\begin{eqnarray*}
\frac{1}{t} \rho\left(\partial_v \varphi \left[ Q^{\alpha}_{d\beta}(\partial_x \bar{\varphi}) Y^\beta(g) \right]\right)&=&
\frac{1}{t} \rho\left(Y'(\varphi) \left[ Q^{\alpha}_{d\beta}(\partial_x \bar{\varphi}) Y^\beta(g) \right]\right) \\
&&\hbox{}-\rho\left(\partial_x(\varphi) \left[ Q^{\alpha}_{d\beta}(\partial_x \bar{\varphi}) Y^\beta(g) \right]\right)\\
&&\hbox{}+\frac{1}{t} \rho\left(\varphi' \partial_x(\varphi) \left[ Q^{\alpha}_{d\beta}(\partial_x \bar{\varphi}) Y^\beta(g) \right]\right)
\end{eqnarray*}
where all terms now have the correct form. 
\end{proof}


Recall that we do not hope to have any good estimate on $Y^{\alpha}(\varphi)$ if $|\alpha|=N$, since its transport equation would then contain a source term of the form $\nabla Z^\alpha Z \phi$ and we only hope to have estimates for $\nabla Z^\beta \phi$ up to $\beta=N$. On the other hand, using directly the Poisson equation satisfied by $Z^\alpha(\phi)$, we will have good estimates on $\partial_x \partial_x Z^\alpha(\phi)$. To take advantage of that fact\footnote{Note that these difficulies arise only at the top order $|\alpha|=N$. An alternative to the approach taken here would be to allow for the top order estimates to grow slightly in $t$.}, we will need later the following technical lemma, which improves upon Lemma \ref{lem:pcm} by describing the structure of the $P$ multilinear forms a little further. 

\begin{lemma} \label{lem:pcmto}
With the notations of Lemma \ref{lem:pcm}, the multilinear forms $P^{\alpha,j}_{d\beta}(\bar{\varphi})$ and $\tilde{P}^{\alpha,j}_{d\beta}(\bar{\varphi})$  can be written as
\begin{eqnarray*}
P^{\alpha,j}_{d\beta}(\bar{\varphi})&=& P_1(\bar{\varphi})+\sum_{\varphi \in \mathcal{M}}\sum_{i=1}^n C_{i,\varphi} Y_{i} Y^{\rho_{i,\varphi}}(\varphi)+ \sum_{\varphi \in \mathcal{M}}\sum_{i=1}^n \mathcal{P}_{i \varphi}(\bar{\varphi}) \partial_{x^i} Y^{\eta_{i,\varphi,\varphi'}}(\varphi), \\
\tilde{P}^{\alpha,j}_{d\beta}(\bar{\varphi})&=&\tilde{P}_1(\bar{\varphi})+\sum_{\varphi \in \mathcal{M}}\sum_{i=1}^n \tilde{C}_{i,\varphi} Y_i Y^{\tilde{\rho}_{i,\varphi}}(\varphi)+ \sum_{\varphi  \in \mathcal{M}}\sum_{i=1}^n \widetilde{\mathcal{P}}_{i\varphi}(\bar{\varphi})  \partial_{x^i} Y^{\tilde{\eta}_{i,\varphi,\varphi'  }}(\varphi)
, 
\end{eqnarray*}
where 
\begin{enumerate}
\item $P_1$ and $\tilde{P}_1$ are multilinear forms of degree $d$ and signature less than $k$ satisfying $k \le |\alpha|-1, k+|\beta| \le |\alpha|$, 
\item $Y_i$ are the modified uniform motions defined at the beginning of Section \ref{se:dmvfa}, i.e. $Y_i=t \partial_{x^i}+ \partial_{v^i}-\sum_{k=1}^n \Phi^{k}_i \partial_{x^k}$,
\item $|\rho_{i,\varphi}|=|\alpha|-1$, $\tilde{\rho}_{i,\varphi}=|\alpha|-1$, $|\eta_{i,\varphi,\varphi'}| \le |\alpha|-1$, $|\tilde{\eta}_{i,\varphi,\varphi'}| \le |\alpha|-1$,  
\item $C_{i,\varphi}$, $\tilde{C}_{i, \varphi}$ are constants and $\mathcal{P}_{i\varphi}(\bar{\varphi})$, $\widetilde{\mathcal{P}}_{i\varphi}(\bar{\varphi})$ are polynomial of degree at most $|\alpha|$ in the $\varphi' \in \mathcal{M}$. 
\end{enumerate}
\end{lemma}


\begin{proof}
This is an easy proof by induction. The case $|\alpha|=1$ has already been proven in the proof of the previous lemma. Assume that the statement of this lemma holds for some $\alpha$ and let $Z \in \Gamma$.
As before, it is sufficient to consider only $Z\left( \rho(Z^\alpha) \right)$. We will only be interested in the top order terms, that is to say terms containing $Y^\eta(\varphi)$ with $\eta=|\alpha|+1$. Note that they can only be generated by applying a vector field to terms containing $Y^{\eta'}(\varphi)$ with $\eta'=|\alpha|$.

Using the induction hypothesis, the top order terms coming from the $P$ multilinear forms will give terms of the form
\begin{equation}
\sum_{\varphi \in \mathcal{M}}\sum_{i=1}^n C_{i,\varphi} Z \left[ Y_{i} Y^{\rho_{i,\varphi}}(\varphi) \right]
\end{equation}
and 
\begin{equation}\label{eq:stt}
\sum_{\varphi, \varphi' \in \mathcal{M}}\sum_{i=1}^n D_{i,\varphi,\varphi'} Z\left[ \varphi' \partial_{x^i} Y^{\eta_{i,\varphi,\varphi'}}(\varphi)\right].
\end{equation}
For the first type of terms, we have 
$$
 Z \left[ Y_{i} Y^{\rho_{i,\varphi}}(\varphi) \right]= (Y+ \varphi'' \partial_x) \left[ Y_{i} Y^{\rho_{i,\varphi}}(\varphi) \right]= Y Y_{i} Y^{\rho_{i,\varphi}}(\varphi) + \varphi'' \partial_x Y_{i} Y^{\rho_{i,\varphi}}(\varphi).  
$$
Now the second term on the right-hand side has the right structure. The first term also has the right structure, since $[Y, Y_i]$ can be written as a linear combinations of terms of the form $\partial_x$, the modified uniform motions $Y_j$ and $\varphi'' \partial_x$ and $Y'(\varphi) \partial_x $.

The terms of the type \eqref{eq:stt} can be treated similarly, the only dangerous terms being of the form  $\varphi' Y \partial_{x^i} Y^{\eta_{i,\varphi,\varphi'}}(\varphi)$, where the $Y$ and $\partial_{x^i}$ can be commuted up to lower order terms. 

The top order terms coming from the $Q$ multi-linear forms will give terms of the form
\begin{eqnarray*}
\rho\left( Z [Y^{\eta} ( \partial_x \varphi) Y^\beta(g)]\right)&=&\rho \left( \left( Y+ \varphi'' \partial_x \right) \left[Y^{\eta}(\partial_x \varphi)  Y^\beta(g)\right] \right) \\
&=& \rho \left(  Y \left[Y^{\eta}(\partial_x \varphi)  Y^\beta(g)\right] \right) 
+  \rho \left(  \varphi'' \partial_x \left[Y^{\eta}(\partial_x \varphi)  Y^\beta(g)\right] \right),
\end{eqnarray*}
where $|\eta|=|\alpha|-1$.
Now the first term on the right-hand side will contribute only the $Q$ forms so it can be ignored here. While the second term, repeating the argument of the previous lemma, gives the following contribution to the $P$ forms  
$$
\frac{1}{t} \rho \left( Y(\varphi)'' Y^\eta(\partial_x \varphi) Y^\beta(g)\right)
$$
which, since $|\eta|\le |\alpha|-1$ is not of top order and 
$$
\frac{1}{t} \rho \left( \varphi'' Y Y^\eta(\partial_x \varphi) Y^\beta(g)\right)
$$
as well as
\begin{eqnarray}
\frac{1}{t} \rho \left( \varphi' \varphi'' \partial_x\left[Y^\eta(\partial_x \varphi) Y^\beta(g)\right]\right)&=&\frac{1}{t} \rho \left( \varphi' \varphi'' \partial_x\left[Y^\eta(\partial_x \varphi)\right] Y^\beta(g)\right)\nonumber\\
&&\hbox{}+
\frac{1}{t} \rho \left( \varphi' \varphi'' Y^\eta(\partial_x \varphi) \partial_x\left[Y^\beta(g)\right]\right)
\end{eqnarray}
which are all of top order. Again, we use that $\partial_x$ essentially commutes with any modified vector field up to lower order terms, to put these last terms in the right form. 


\end{proof}

The following commutation property will also be usefull later.

\begin{lemma}\label{lem:cdxy}
For any multi-index $\alpha$ and any $1 \le i \le 3$, we have 
$$
[\partial_{x^i}, Y^\alpha ]= \sum_{d=0}^{|\alpha|} \,\, \sum_{\substack{|\beta| \le |\alpha|-1,\\ 1 \le j \le 3}} P_{i,d\beta}^{\alpha,j}(\partial_x \bar{\varphi}) Y^\beta \partial_{x^j},
$$
where the $P_{i,d\beta}^{\alpha,j}(\partial_x \bar{\varphi})$ are multilinear forms of degree $d$ of the form \eqref{eq:qmf} with a signature less than $k$ such that $k+|\beta| \le |\alpha|-1$ .
\end{lemma}
\begin{proof}We prove the $|\alpha|=1$ case, the general case following by an easy induction argument. Let $Y$ be modified vector field. We will write schematically $Y=Z+\varphi \partial_x$, where $Z$ is a non-modified vector field and $\varphi \partial_x$ stands for a linear combination of products of some $\varphi \in \mathcal{M}$ and some $\partial_{x^k}$. For simplicity, assume that $Z$ commutes with $\partial_{x^i}$ (the cases where $Z$ do not commute with $\partial_{x^i}$ can be treated similarly, since the resulting error terms do not involve any terms depending on the coefficients in $\mathcal{M}$.)
An easy computation then shows that 
$$[ \partial_{x^i}, Y]=[ \partial_{x^i}, \varphi \partial_x]= \partial_{x^i} (\varphi) \partial_x,$$
which is of the desired form. 
\end{proof}
Finally, let us also remark that
\begin{lemma} \label{lem:tynzp}
For any multi-index $\alpha$, 
$$
Y^\alpha \nabla \phi= Z^\alpha \nabla \phi +\frac{1}{t}\sum_{d=1}^{|\alpha|} P^\alpha_{d\beta} ( \bar{\varphi} ) Z^\beta \nabla \phi,
$$
where $P^\alpha_{d\beta} ( \bar{\varphi} )$ are multilinear forms of degree $d$ and signature less than $k$ such that $k \le |\alpha|-1$ and $k+|\beta| \le |\alpha|$.
\end{lemma}
\begin{proof}
We only do the $|\alpha|=1$ case, the rest of the proof being an easy induction. Let $Y'$ be a modified vector field and write schematically $Y'=Z'+\varphi \partial_x$. We have
$$
Y' \nabla \phi=\left(Z' +\varphi \partial_x\right)\nabla \phi=Z'\nabla \phi+\frac{\varphi}{t}[t\partial_x]\nabla \phi,
$$
which is of the correct form since $t\partial_x \in \Gamma$.
\end{proof}
\subsection{Bootstrap assumptions and beginning of the proof}
Let now $f$ be a solution to the Vlasov-Poisson system in dimension $n=3$ such that the hypotheses of Theorem \ref{th:mt} are satisified.

We consider the following bootstrap assumptions. Let $T \ge 0$ be the largest time so that, for all $t \in [0,T]$ and all $x \in \mathbb{R}^3$,
\begin{enumerate}
\item 
\begin{equation}\label{3dbaseb}
E_{N,\delta}[f(t)] \le 2 \epsilon,
\end{equation}
\item For all $0 < \delta' \le \delta$, there exists a $C_{\delta'} > 0$ such that for all multi-index $\alpha$ with $|\alpha| \le N$, 
\begin{equation}\label{3dbasp2d}
||\nabla^2 Z^\alpha \phi (t) ||_{L^{1+\delta'}} \le C_{\delta'} \epsilon^{1/2}.
\end{equation}
\item For all multi-index $\alpha$ with $|\alpha| \le N-(9/2+1)$, 
\begin{equation}\label{3dbasp}
| \nabla Z^\alpha\phi (t,x) |  \lesssim \frac{\epsilon^{1/2}}{1+t^2}
\end{equation}
and 
\begin{equation}
| Y^\alpha \nabla \phi (t,x)| \lesssim \frac{\epsilon^{1/2}}{1+t^2} \label{3dbaspy}
\end{equation}
\item For all multi-index $\alpha$ with $|\alpha| \le N-(9/2+2)$, we have for all $\varphi \in \mathcal{M}$, 
\begin{equation} \label{bas2}
\left| Y^\alpha (\varphi)(t,x,v) \right| \lesssim \epsilon^{1/2} \left|1+  \log (1+t) \right|,
\end{equation}
\item For all $\varphi \in \mathcal{M}$, all $1 \le i \le 3$ and all multi-index $|\alpha| \le N-(9/2+3)$ ,  
\begin{equation} \label{bas3}
\left|  \partial_{x^i} Y^\alpha (\varphi)(t,x,v) \right| \lesssim  \epsilon^{1/2} ,
\end{equation}
\end{enumerate}
It follows from the initial data assumption and standard arguments that $T>0$. In the rest of the proof, we will try to improve each of the above assumptions, which would show that $T=+\infty$. 
 
Note that in view of Lemma \ref{lem:cdxy}, \eqref{bas3} is equivalent to 
$$
\left| Y^\alpha \partial_{x^i}  (\varphi)(t,x,v) \right| \lesssim  \epsilon^{1/2}
$$
and we will switch freely between the two in the rest of the article. Note moreover than in view of the Gagliardo- Niremberg inequality, assumption \eqref{3dbasp2d} immediately implies that 
\begin{equation} \label{es:lqp}
|| \nabla Z^\alpha \phi ||_{L^{q}} \le C_q\, \epsilon^{1/2},
\end{equation}
with $q=3/2 < \frac{3(1+\delta')}{2-\delta'} < 2$, in view of the definition of $\delta$. In particular, $q$ can be taken as close to $3/2$ as wanted. 

Finally, the notation $A \lesssim B$ used in \eqref{3dbasp}, \eqref{3dbaspy}, \eqref{bas2}, \eqref{bas3} stand for $A \le C_{N, \delta, n}B$ where $C_{N, \delta, n}$ is a constant depending only on $N,n,\delta$ (with here $n=3$). We will eventually improve each of these inequalities by replacing the $\epsilon^{1/2}$ on the right-hand sides of \eqref{3dbasp}, \eqref{3dbaspy}, \eqref{bas2}, \eqref{bas3} by $\epsilon$, choosing $\epsilon$ sufficiently small to absorb the $C_{N, \delta, n}$.
\subsection{Klainerman-Sobolev inequalities with modified vector fields}\label{se:ksimvf}
Using the bootstrap assumptions above, we have

\begin{proposition} \label{es:ksmc}
For all multi-index $|\alpha| \le N-3$, 
\begin{eqnarray*}
\rho\left(\left| Y^\alpha f(t) \right| \right) &\lesssim& \frac{1}{\left(1+t+|x| \right)^3}\sum_{|\beta| \le |\alpha|+3} || Y^\beta f(t)||_{L^1(\mathbb{R}^3_x \times \mathbb{R}^3_v)},\\
  &\lesssim& \frac{1}{\left(1+t+|x| \right)^3}\sum_{|\beta| \le N} || Y^\beta f(t)||_{L^1(\mathbb{R}^3_x \times \mathbb{R}^3_v)}.
\end{eqnarray*}
\end{proposition}
\begin{proof}
Similarly to the proof of Proposition \ref{prop:gksiva1}, let us fix $(t,x) \in \mathbb{R} \times \mathbb{R}^3$ and let $\widetilde{\psi}$ be defined by 
$\widetilde{\psi}: B_{3}(0,1/2) \ni y \rightarrow \rho\left( \left| Y^\alpha f\right| \right) \left(t,x+(t+|x|)y \right).$
We fix $\delta'=\frac{1}{12}$ and apply a $1d$ Sobolev inequality
$$
 \rho(| Y^\alpha f| )(t,x) \lesssim \int_{|y_1| \le {\delta'}^{1/2}}\left( \left| \partial_{y_1} \widetilde{\psi} \right|+ |\widetilde{\psi}| \right)(y_1,0,0) dy_1,
$$
where as before 
\begin{eqnarray*}
\partial_{y_1}\widetilde{\psi}(y)&=&(t+|x|) \partial_{x_1}\rho(| Y^\alpha f| )\left(t,x+(t+|x|)y \right) \\
&=& t \int_{v} \partial_{x_1}\left(| Y^\alpha f| \right)\left(t,x+(t+|x|)y,v \right) dv+ |x|\int_{v} \partial_{x_1}\left(| Y^\alpha f|\left(t,x+(t+|x|)y,v \right) \right) dv .
\end{eqnarray*}
Now, 
\begin{eqnarray*}
t \int_{v} \partial_{x_1}\left(| Y^\alpha f| \right) dv&=& \int_{v}\left[\left( t\partial_{x_1}+\partial_{v_1}-\sum_{i=1}^n \Phi_1^j \partial_{x^j}  + \sum_{j=1}^n \Phi_1^j \partial_{x^j}\right)\left(| Y^\alpha f| \right)\right] dv  \\
&= &\int_{v} Y_1 \left( | Y^\alpha f| \right) dv  + \int_v \sum_{j=1}^n \Phi_1^j \partial_{x^j}\left(| Y^\alpha f| \right) dv.
\end{eqnarray*}
The first term on the right-hand side is simply estimated by 
$$\left| \int_{v} Y_1 \left( | Y^\alpha f| \right) dv \right| \le \int_{v} | Y_1 Y^\alpha f| dv.$$
For the second term, we again try to force the apparition of our modified vector fields
\begin{eqnarray*}
\int_v  \Phi_1^j \partial_{x^j}\left(| Y^\alpha f| \right) dv&=&\int_v \frac{\Phi_1^j}{t} \left( t  \partial_{x^j}+ \partial_{v^j}- \sum_{k=1}^n \Phi^k_j \partial_{x^k} \right) \left(| Y^\alpha f| \right)dv \\
&&\hbox{}- \int_v \frac{\Phi_1^j}{t}\left( \partial_{v^j}- \sum_{k=1}^n \Phi^k_j \partial_{x^k} \right) \left(| Y^\alpha f| \right)dv \\
&=&\int_v \frac{\Phi_1^j}{t} Y_j \left(| Y^\alpha f| \right)dv \\
&&\hbox{}- \int_v \frac{\Phi_1^j}{t}\left( \partial_{v^j}- \sum_{k=1}^n \Phi^k_j \partial_{x^k} \right) \left(| Y^\alpha f| \right)dv.
\end{eqnarray*}
The first term on the right-hand side can then be estimated as above, using that $|\frac{\Phi_1^j}{t}|$ is uniformly bounded from the bootstrap assumptions \eqref{bas2}. 
For the remainder terms, we first note than in view of the bootstrap assumptions \eqref{bas2}, the terms of the form 
$$
 \int_v \frac{\Phi_1^j}{t}\Phi^k_j \partial_{x^k}  \left(| Y^\alpha f| \right)
$$
can be estimated by 
$$
 \int_v \left|   \partial_{x^k} \left(| Y^\alpha f| \right) \right| dv.
$$
For the last type of terms, we integrate by parts in $v$

\begin{eqnarray*}
\int_v \frac{\Phi_1^j}{t} \partial_{v^j} \left(| Y^\alpha f| \right)dv=-\int_v \frac{\partial_{v^j}\Phi_1^j}{t}  \left(| Y^\alpha f| \right)dv.
\end{eqnarray*}
We now rewrite $\partial_{v^j}\Phi_1^j$ as
\begin{eqnarray*}
\partial_{v^j}\Phi_1^j&=& \left( t \partial_{x^j} + \partial_{v^j}-\sum_{k=1}^n \Phi_j^k \partial_{x^k} \right)\Phi_1^j - \left( t \partial_{x^j}-\sum_{k=1}^n \Phi_j^k \partial_{x^k} \right) \Phi_1^j \\
&=& Y_j (\Phi_1^j)- \left( t \partial_{x^j}-\sum_{k=1}^n \Phi_j^k \partial_{x^k} \right) \Phi_1^j.
\end{eqnarray*}
The first term only grow like $|1+ \log(1+ t)|$ according to \eqref{bas2} and this growth can be absorbed thanks to the $1/t$ factor.  
For the second term, using \eqref{bas3} and \eqref{bas2}
$$
\left| t \partial_{x^j}  (\Phi_1^j) \right|+ \left| \Phi_j^k \partial_{x^k} ( \Phi_1^j) \right| \le \epsilon^{1/2} t + \epsilon^{1/2} |1+ \log(1+t)|,
$$
and again we can absorb the growth using the $1/t$ factor. 

Putting everything together we have obtained that
\begin{eqnarray*}
\rho(| Y^\alpha(f)|)(t,x) &\lesssim& 
\int_{|y_1| \le {\delta'}^{1/2} } \int_v \left(  |Y Y^\alpha(f) | + |Y^\alpha(f)| \right)\left(t,x+(t+|x|)(y_1,0,0) \right) dv dy_1.
\end{eqnarray*}
The remaining of the proof follows as in the proof of \eqref{prop:gksiva1}, repeating the previous arguments for each of the variables and applying the usual change of coordinates.
\end{proof}

\subsection{Estimates on products of type $Y^{\alpha}(\varphi) Y^{\beta}(f)$} \label{se:ep}
Due to the form of the commutators of Section \ref{se:icf}, we will need to estimate terms of the form $Y^{\alpha}(\varphi) Y^{\beta}(f)$. When $\alpha$ is sufficiently small, we will have access to pointwise estimates on $Y^{\alpha}(\varphi)$ so there is no difficulty. When $\alpha$ is large, say $\alpha=N$, we have for the moment no estimate on $Y^{\alpha}(\varphi)$ and we can certainly not hope to prove pointwise estimates for these quantities, because, in view of the transport equations satisfied by the coefficients $\varphi$, these estimates would in turn require pointwise estimates on $\nabla Y^\alpha (\phi)$, and these estimates do not hold at the top order. Instead, we will prove directly estimates on the products $Y^{\alpha}(\varphi) Y^{\beta}(f)$, taking advantage of the fact that $Y^{\beta}(f)$ are integrable in $v$. More precisely, 

\begin{proposition}\label{prop:ep}
 Let $\sigma> 0$. Then, there exists a $C_{\sigma}> 0$ such that for any multi-indices $\alpha, \beta$, $|\alpha| \le N-1$, $|\beta| \le 9/2 +3$ and any $\varphi \in \mathcal{M}$, we have for all $t \in [0,T]$ and all $1 \le i \le 3$,

\begin{eqnarray*}
|| Y^{\alpha}(\varphi)(t) Y^{\beta}(f)(t) ||_{L^1 \left( \mathbb{R}_x \times \mathbb{R}_v \right)} &\le& C_{\sigma} (1+t)^{\sigma} \epsilon, \label{eq:pres}\\
||Y_i Y^{\alpha}(\varphi)(t) Y^{\beta}(f)(t) ||_{L^1 \left( \mathbb{R}_x \times \mathbb{R}_v \right)} &\le& C_{\sigma} (1+t)^{\sigma} \epsilon, \\
||\partial_{x^i} Y^{\alpha}(\varphi)(t) Y^{\beta}(f)(t) ||_{L^1 \left( \mathbb{R}_x \times \mathbb{R}_v \right)} &\le& C_{\sigma} (1+t)^{\sigma} \epsilon, 
\end{eqnarray*}
\end{proposition}
\begin{proof}
Let $\sigma > 0$ and let $\alpha$ be a multi-index satisfying $|\alpha|\le N$ and such that if $|\alpha|=N$ then $Y^\alpha =Y_j Y^{\alpha'}$ or $Y^\alpha = \partial_{x^j} Y^{\alpha'}$ with $|\alpha'|=N-1$.

We have 
$$
T_\phi\left(Y^{\alpha}(\varphi) Y^{\beta}(f) \right)= T_\phi\left( Y^\alpha(\varphi) \right)  Y^\beta(f)+Y^\alpha(\varphi) T_\phi \left( Y^{\beta}(f) \right)=I_1+I_2,
$$
where $I_1= T_\phi\left( Y^\alpha(\varphi) \right) Y^\beta(f)$ and $I_2=Y^\alpha(\varphi) T_\phi \left( Y^{\beta}(f) \right)$.
In view of Lemma \ref{lem:cl}, it suffices to show that 
$$||I_1, I_2||_{L^1_{x,v}} \le C_\sigma (1+t)^{1-\sigma}.$$

\begin{enumerate}
\item Estimates on $I_1$ \\
We have 
\begin{eqnarray*}
I_1&=&[T_\phi, Y^\alpha](\varphi) Y^\beta(f)+ Y^{\alpha} [T_\phi( \varphi)] Y^\beta(f) \\
&=& I_{1,1}+I_{1,2},
\end{eqnarray*}
with $I_{1,1}=[T_\phi, Y^\alpha](\varphi)  Y^\beta(f)$ and 
$I_{1,2}=Y^{\alpha}[T_\phi( \varphi)] Y^\beta(f)$.
For $I_{1,1}$, we use the commutation formula \eqref{eq:gcfm}
$$I_{1,1}=\sum_{d=0}^{|\alpha|+1}\sum_{i=1}^n \sum_{\substack{|\gamma| \le |\alpha|,\\ |\eta| \le |\alpha|}} P^{\alpha,i}_{d \gamma \eta}(\bar{\varphi}) \partial_{x^i} Z^\gamma(\phi) Y^\eta(\varphi)Y^\beta(f) ,$$
where $P^{\alpha,i}_{d \gamma \eta}$ satisfies the requirement of Lemma \ref{lem:gcfm}, in particular, it has signature less than $k$ such that $k\le |\alpha|-1$ and  $k +|\gamma|+|\eta| \le |\alpha|+1$. \\

Case 1: $|\gamma| \le N-\left( 9/2 +1 \right)$. \\
It then follows from the bootstrap assumption \eqref{3dbasp}, that we have the pointwise estimate
$$ \left| \partial_{x^i} Z^\gamma(\phi) \right| \le \frac{\epsilon^{1/2}}{1+t^2}.$$
Moreover, since $k + |\eta| \le |\alpha|+1 \le N+1$, either $k \le N-(9/2+2)$, and we have access to the pointwise estimate
$$
\left| P^{\alpha,i}_{d \gamma \eta}(\bar{\varphi}) \right| \le \epsilon^{d/2}\left( 1+ \log(1+t) \right)^{d}
$$
or $k > N-(9/2+2)$. In this case, we have $|\eta| \le N+1-k \le N-(9/2+2)$, since $N \ge 14$, so that we now have access to pointwise estimates on $Y^\eta(\varphi)$. Note also that since $k \le N+1$, there is at most one factor in each of products of $Y^{\rho_j}(\bar{\varphi})$ in the decomposition of $P^{\alpha,i}_{d \gamma \eta}(\bar{\varphi})$ for which we do not have access to pointwise estimates. In conclusion, it follows that we have an estimate of the form 
$$
\left| P^{\alpha,i}_{d \gamma \eta}(\bar{\varphi}) \partial_{x^i} Z^\gamma(\phi) Y^\eta(\varphi) \right| \lesssim \epsilon^{1/2} \frac{\left(1+ \log(1+t) \right)^d}{1+t^2} \sum_{|\rho| \le |\alpha|-1} |Y^\rho(\varphi)|.
$$
Case 2: $\gamma > N-\left( 9/2 +1 \right)$.\\

Then, $k+|\eta| \le N+1-|\gamma| \le N-(9/2+2)$ since $N \ge 14$ and we can bound $P^{\alpha,i}_{d \gamma \eta}(\bar{\varphi})$ and $Y^{\eta}(\varphi)$ pointwise. Since $|\beta| \le 9/2+3$, we have access to pointwise bound on $\rho(|Z^\beta(f)|)$, so that we can estimate 
\begin{eqnarray*}
&&|| P^{\alpha,i}_{d \gamma \eta}(\bar{\varphi}) \partial_{x^i} Z^\gamma(\phi) Y^\eta(\varphi)Y^\beta(f) ||_{L^1_{x,v}} \lesssim \\
&&\quad \quad \quad \quad \quad \quad \quad   \quad  \quad   \quad   \epsilon^{1/2}\left( 1+ \log(1+t) \right)^{d+1} ||\partial_{x^i} Z^\gamma(\phi)||_{L^q_x} || \rho\left( |Y^\beta(f)| \right) ||_{L^p_x},
\end{eqnarray*}
where $1/q+1/p=1$.
Taking $q$ as in \eqref{es:lqp} with $\delta'$ as small as needed (depending only on $\sigma$), we obtain, using the pointwise estimates of Proposition \ref{es:ksmc}, that 
$$|| \rho\left( |Y^\beta(f)| \right) ||_{L^p_x} \lesssim C_\sigma \frac{\epsilon}{(1+t)^{2-\sigma}},$$ so that  
$$
|| P^{\alpha,i}_{d \gamma \eta}(\bar{\varphi}) \partial_{x^i} Z^\gamma(\phi) Y^\eta(\varphi)Y^\beta(f) ||_{L^1_{x,v}} \lesssim \epsilon^{2}\frac{\left( 1+ \log(1+t) \right)^{d+1}}{(1+t)^{2-\sigma}},
$$
which, assuming $\sigma <1$, is integrable in $t$.\\

We now turn to the estimates on $I_{1,2}$. Recall that we have $T_\phi(\varphi)=t \partial_{x^i} Z(\phi)$ for some $Z$ and some $x^i$, unless $Z$ is the spatial scaling vector field, in which case  $T_\phi(\varphi)=t \partial_{x^i} \left( Z(\phi)-2\phi \right)$. Since the extra term can be handled similarly, we will only treat the case of the non spatial scaling vector fields below.
Applying Lemma \ref{lem:tynzp}, we have
\begin{eqnarray} 
Y^\alpha \left( T_\phi(\varphi) \right)&=& t Y^\alpha \partial_{x^i} Z(\phi) \nonumber \\
&=& t Z^\alpha \partial_{x^i} Z(\phi) +t \frac{1}{t}\sum_{d=1}^{|\alpha|} P^\alpha_{d\eta} ( \bar{\varphi} ) Z^\eta \partial_{x^i} Z(\phi), \nonumber \\
&=& t Z^\alpha \partial_{x^i} Z(\phi) +\sum_{d=1}^{|\alpha|} P^\alpha_{d\eta} ( \bar{\varphi} ) Z^\eta \partial_{x^i} Z(\phi),\label{eq:inter}
\end{eqnarray}
where $P^{\alpha}_{d \eta}(\bar{\varphi})$ is a multi-linear form of degree $d$ and signature less than $k$ with $k \le |\alpha|-1$ and $k+|\eta| \le |\alpha|$.

Assume first that $|\alpha| \le N-1$. 
For the first term on the right-hand side of \eqref{eq:inter}, we have 
$$
 |t Z^\alpha \partial_{x^i} Z(\phi)| \lesssim t \sum_{|\eta| \le |\alpha|} |\partial_{x^i} Z^\eta Z(\phi) |.
$$
Since $|\alpha| \le N-1$, we have $|\eta|+1 \le N$ in the above sum. Thus, $|t Z^\alpha \partial_{x^i} Z(\phi)| | Z^\beta(f) |$ can be estimated as before using the H\"older inequality, pointwise estimates on $\rho( | Z^\beta(f) | )$ and the estimate \eqref{es:lqp}. 

If now $|\alpha|=N$, then by assumption, $Y^\alpha= Y^jY^{\alpha'}$ or $Y^\alpha= \partial_{x^j}Y^{\alpha'}$ with $|\alpha'|=N-1$, so that $Z^\alpha=t \partial_{x^i} Z^{\alpha'}$ or $Z^\alpha= \partial_{x^i} Z^{\alpha'}$. Thus, 

\begin{eqnarray*}
|t Z^\alpha \partial_{x^i} Z(\phi)| &\lesssim& t(1+t) \sum_{|\eta| \le N-1} |\partial_x \partial_{x^i} Z^\eta Z(\phi)  |, \\
 &\lesssim& t(1+t) \sum_{|\eta| \le N} |\partial^2_x Z^\eta (\phi)  |.
\end{eqnarray*}
We can now estimate $|t Z^\alpha \partial_{x^i} Z(\phi)| \rho( | Z^\beta(f) | )$ using H\"older inequality with $p=\frac{3}{3-\sigma}=1+\sigma'$ with $\sigma'=\frac{\sigma}{3-\sigma}> 0$ assuming $\sigma < 3$ and $q=\frac{3}{\sigma}$. Recall that $|| \partial^2_x Z^\eta (\phi) ||_{L^p}$ is bounded thanks to the boostrap assumption \eqref{3dbasp2d} provided $\sigma$ is sufficiently small. Moreover, using the pointwise estimate on $\rho(| Z^\beta(f) | )$, we have 

$$
|| \rho(| Z^\beta(f) |) ||_{L^q(\mathbb{R}^3_x)} \lesssim \epsilon (1+t)^{3-\sigma}.
$$

Whether $|\alpha|\le N-1$ or $|\alpha|=N$, the above estimates gives
$$
\int_{x,v} |t Z^\alpha \nabla Z(\phi)|| Z^\beta(f) | dxdv\le C_\sigma \frac{ \epsilon^{3/2}}{(1+t)^{1-\sigma}}
$$
which, after integration, gives rise to the $t^\sigma$ growth in the statement of the proposition.

For the second term on the right-hand side of \eqref{eq:inter}, we have either $|\eta|+1 \le N-(9/2+1)$, in which case, we have access to the pointwise estimate 
$$|\partial_{x^i} Z^\eta Z(\phi) | \lesssim \frac{\epsilon^{1/2}}{(1+t)^2}$$ and thus, 
$$
P^\alpha_{d\eta}(\bar{\varphi}) |\partial_{x^i} Z^\eta Z(\phi) | Z^{\beta}(f)| \lesssim  \frac{\epsilon^{1/2}\left(1+\log(1+t) \right)^{d}}{(1+t)^2} \sum_{|\rho| \le N-1} |Y^\rho(\varphi)|| Z^{\beta}(f)|,
$$
where we have used the fact that there is at most one term in $P^\alpha_{d\eta}(\bar{\varphi})$ for which we do not have access to pointwise estimates, 
or we have $|\eta|+1 > N-(9/2+1)$, in wich case $k \le N-(9/2+2)$ and we can bound pointwise $P_{d \eta}(\bar{\varphi})$ as
$$
| P_{d \eta}(\bar{\varphi}) | \lesssim \left( 1+\log (1+t) \right)^{d}.
$$
$P^\alpha_{d\eta}(\bar{\varphi}) |\partial_{x^i} Z^\eta Z(\phi) | Z^{\beta}(f)|$ can then be estimated as before, using H\"older inequality, pointwise estimates on $\rho( | Z^\beta(f) | )$ and the estimate \eqref{es:lqp}.
\item Estimates on $I_2$. Using \eqref{eq:gcfm} again, we have
$$
| I_2 | \lesssim \sum_{d=0}^{|\beta|+1}\sum_{i=1}^n \sum_{\substack{|\gamma| \le |\beta|,\\ |\eta| \le |\beta|}} | P^{\beta,i}_{d \gamma \eta}(\bar{\varphi}) | \, |\partial_{x^i} Z^\gamma(\phi)| \, |Y^\eta(f)| \, |Y^\alpha(\varphi)|,
$$
where $P^{\beta,i}_{d\gamma \eta}(\bar{\varphi})$ are multilinear forms of degree $d$ and signature $k \le |\beta|-1$ satisfying $k+|\gamma|+|\eta| \le |\beta|+1$.
Since $|\beta| \le 9/2+3$, we have $|\beta|-1 \le 9/2+2 \le N-(9/2+2)$. Thus, using the bootstrap assumption \ref{bas2}, all the terms in the decomposition of each of the $P^{\beta,i}_{d\gamma \eta}(\bar{\varphi})$ can be bounded pointwise. Thus, we have 
$$
| I_2 | \lesssim \left(1+\log(1+t) \right)^{|\beta|+1}\sum_{i=1}^n \sum_{\substack{|\gamma| \le |\beta|,\\ |\eta| \le |\beta|, \\ |\gamma|+|\eta| \le |\beta|+1}} |\partial_{x^i} Z^\gamma(\phi)| \, |Y^\eta(f)| \, |Y^\alpha(\varphi)|.
$$
Now since $|\gamma| \le |\beta| \le 9/2 +3$ in the above sum, we have $|\gamma| \le N-(9/2+1)$ since $N \ge 14$ and thus, we can bound $|\partial_{x^i} Z^\gamma(\phi)|$ pointwise using the boostrap assumption \eqref{3dbasp}.
We have thus obtained
$$
||I_2(t)||_{L^1(\mathbb{R}^n_v) \times L^2(\mathbb{R}^n_x)} \lesssim \frac{\epsilon^{1/2}}{(1+t)^{1+\delta'}}\sum_{|\eta| \le |\beta|}|| Y^{\alpha}(\varphi)(t) Y^{\eta}(f)(t) ||_{L^1 \left( \mathbb{R}_x \times \mathbb{R}_v \right)},
$$
for some $\delta'> 0$.
\item Conclusions of the proof of the lemma:  \\
Let 
\begin{eqnarray*}
F_1(t)&=&\sum_{|\eta| \le |\beta|}\sum_{|\alpha| \le N-1}|| Y^{\alpha}(\varphi)(t) Y^{\eta}(f)(t) ||_{L^1 \left( \mathbb{R}_x \times \mathbb{R}_v \right)}, \\
F_2(t)&=&\sum_{|\eta| \le |\beta|}\sum_{i=1}^3 \sum_{|\alpha| = N-1}||Y_iY^{\alpha}(\varphi)(t) Y^{\eta}(f)(t) ||_{L^1 \left( \mathbb{R}_x \times \mathbb{R}_v \right)}, \\
F_3(t)&=&\sum_{|\eta| \le |\beta|}\sum_{i=1}^3 \sum_{|\alpha| = N-1}||\partial_{x^i} Y^{\alpha}(\varphi)(t) Y^{\eta}(f)(t) ||_{L^1 \left( \mathbb{R}_x \times \mathbb{R}_v \right)}, \\
F&=&F_1+F_2+F_3. 
\end{eqnarray*}
Combining all the ingredients above, We have obtained that, there exists some $\delta'> 0$ such that 
$$
F(t) \lesssim \int_{0}^t \frac{\epsilon^{1/2}}{(1+s)^{1+\delta'}}F(s) ds + C_\sigma \epsilon^{3/2} t^\sigma.$$
Applying Gronwall inequality and using the smallness of the initial data then finishes the proof.
\end{enumerate}
\end{proof}

Using \eqref{eq:Ewv}, we have similarly
\begin{proposition}\label{prop:vwpe}
 Let $\sigma> 0$. Then, there exists a $C_{\sigma}> 0$ such that for any multi-indices $\alpha, \beta$, $|\alpha| \le N-1$, $|\beta| \le 9/2 +3$ and any $\varphi \in \mathcal{M}$, we have for all $t \in [0,T]$, and all $1 \le i \le 3$,

\begin{eqnarray}
||(1+v^2)^{\frac{\delta(\delta+3)}{2(1+\delta)}} Y^{\alpha}(\varphi)(t) Y^{\beta}(f)(t) ||_{L^{1+\delta} \left( \mathbb{R}_x \times \mathbb{R}_v \right)} \le C_{\sigma} (1+t)^{\sigma} \epsilon\label{eq:pres2},\\
||(1+v^2)^{\frac{\delta(\delta+3)}{2(1+\delta)}} Y_i Y^{\alpha}(\varphi)(t) Y^{\beta}(f)(t) ||_{L^{1+\delta} \left( \mathbb{R}_x \times \mathbb{R}_v \right)} \le C_{\sigma} (1+t)^{\sigma} \epsilon\label{eq:pres2bis},\\
||(1+v^2)^{\frac{\delta(\delta+3)}{2(1+\delta)}} \partial_{x^i} Y^{\alpha}(\varphi)(t) Y^{\beta}(f)(t) ||_{L^{1+\delta} \left( \mathbb{R}_x \times \mathbb{R}_v \right)} \le C_{\sigma} (1+t)^{\sigma} \epsilon.\label{eq:pres2ter}
\end{eqnarray}
\end{proposition}
\begin{proof}The proof is almost identical to the proof of the previous proposition and therefore left to the reader.
\end{proof}

Finally, we can get rid of the small growth provided with look at a product of the form $Y^\alpha \partial_x (\varphi)Y^\beta(f)$.

\begin{proposition}\label{prop:perg} For any multi-indices $\alpha, \beta$, $|\alpha| \le N-1$, $|\beta| \le 9/2 +3$ and any $\varphi \in \mathcal{M}$, we have for all $t \in [0,T]$ and all $1 \le j \le 3$,

\begin{eqnarray*}
|| Y^{\alpha}(\partial_{x^j}\varphi)(t) Y^{\beta}(f)(t) ||_{L^1 \left( \mathbb{R}_x \times \mathbb{R}_v \right)} &\lesssim&  \epsilon,\\
\end{eqnarray*}
as well as 
\begin{eqnarray*}
||(1+v^2)^{\frac{\delta(\delta+3)}{2(1+\delta)}} Y^{\alpha}(\partial_{x^j}\varphi)(t) Y^{\beta}(f)(t) ||_{L^{1+\delta} \left( \mathbb{R}_x \times \mathbb{R}_v \right)} \lesssim \epsilon.
\end{eqnarray*}
\end{proposition}
\begin{proof}
First note that the previous arguments used in the proof of Proposition \ref{prop:ep} still apply and that we may only focus on the terms leading to the $t^\sigma$ growth in the proof of Proposition \ref{prop:ep}, which were contained in the error term $I_{1,2}$. More precisely, the term leading to the $t^\sigma$ growth is the first one on the right-hand side of \eqref{eq:inter}, i.e. the term $t Z^\alpha \partial_{x^i} Z(\phi)$. In our case, this term should be replaced by 

$$t Z^\alpha \partial_{x^i} \partial_{x^j}Z(\phi),$$ 

with $|\alpha| \le N-1$. 
Commuting the $\partial_{x}$ and $Z^\alpha$, we have 
$$
|t Z^\alpha \partial_{x^i} \partial_{x^j}Z(\phi)| \le t \sum_{|\eta| \le N} |\partial_x^2 Z^\eta \phi|.
$$



We can then repeat the previous arguments (i.e. use H\"older inequality and the bootstrap assumption \eqref{3dbasp2d}), except that we have gained a $1/t$ factor. This gain now means that the resulting error will decay like $1/t^{2-\sigma}$, which is integrable in $t$ and therefore does not lead to any growth. The $L^{1+\delta}_{x,v}$ weighted estimates can be treated similarly. 
\end{proof}

\subsection{Improving the bootstrap assumptions}
We are now in a position to improve each of the boostrap assumptions. 
\subsubsection{Improving the estimates on $Z^\alpha(\phi)$}
Similarly to Lemma \ref{lem:zpes}, we can improve assumption \eqref{3dbasp2d} to
\begin{lemma}
For all $0 < \delta' \le \delta$, there exists a $C_{\delta'} > 0$ such that for all multi-index $\alpha$ with $|\alpha| \le N$, 
\begin{equation}
||\nabla^2 Z^\alpha \phi (t) ||_{L^{1+\delta'}} \le C_{\delta'} \epsilon.
\end{equation}


Applying the Gagliardo-Nirenberg inequality, we deduce that for all $t\in [0,T]$,
$$
||\nabla Z^\alpha \phi(t) ||_{L^q(\mathbb{R}^n)} \lesssim  \epsilon,  $$ 
for all $3/2 < q < \frac{3(1+\delta)}{2-\delta}$. 
\end{lemma}
\begin{proof}
Using the commuted equation \eqref{eq:cdz} for $Z^\alpha \phi$, we have
$$
|| \Delta Z^\alpha(\phi)||_{L^{1+\delta'}} \lesssim J_1 +J_2
$$
where 
\begin{eqnarray*}
J_1=\sum_{j=1}^{|\alpha|} \sum_{d=1}^{|\alpha|+1} \sum_{|\beta| \le |\alpha | } \frac{1}{t^j} ||  \rho\left( P^{\alpha,j}_{d\beta}(\bar{\varphi}) Y^\beta(g) \right)||_{L^{1+\delta'}},
\end{eqnarray*}
where the $P^{\alpha,j}_{d\beta}(\bar{\varphi})$ are multilinear forms of degree $d$ and signature less than $k$ satisfying 
$$k \le |\alpha|, \quad k+|\beta| \le |\alpha|,$$ and, in view of Lemma \ref{lem:pcmto}, such that when $|\alpha|=N$ the only top order terms in $P^{\alpha,j}_{d\beta}(\bar{\varphi})$ are of the form $Y_i Y^\beta(\varphi)$ or $\partial_{x^i} Y^\beta(\varphi)$ with $|\beta|=N-1$, so that we can apply Propositions \ref{prop:ep} and \ref{prop:vwpe},\\

and where
$$
J_2= \sum_{d=0}^{|\alpha|}\sum_{|\beta| \le |\alpha | } || \rho\left( Q^{\alpha}_{d\beta}(\partial_x \bar{\varphi}) Y^\beta(g) \right)||_{L^{1+\delta'}},
$$
where the $Q^{\alpha}_{d\beta}(\partial_x\bar{\varphi})$ are multilinear forms of degree $d$ of the form \eqref{eq:qmf} and signature less than $k'$ satisfying

$$k' \le |\alpha|-1, \quad k'+|\beta| \le |\alpha|.$$ 
Following the strategy of the proof of Lemma \ref{lem:zpes}, we see that it is sufficient to prove $L^p_x$ bounds on $\rho(J_1)$ and $\rho(J_2)$. With this in mind, recall that, for $i=1,2$,
$$
||\rho(J_i) ||_{L^p_x}=\left|\left| \int_v J_i dv\right|\right|_{L^p_x} \lesssim \left( \int_{v} \chi(v)^{-1} dv \right)^{1/q} \left|\left|  \chi(v)^{1/q} J_i \right|\right|_{L^p_{x,v}},
$$
for any weight function $\chi(v)$. Choosing $\chi(v)$, $p$ and $q$ as in the proof of \ref{lem:zpes}, we only need to prove the $v$-weighted $L^p_{x,v}$ bounds for $J_1$ and $J_2$. 
\begin{enumerate}
\item Estimates on $J_1$. Since $k \le |\alpha|$, there can be at most one term in the decomposition of each of the $P^{\alpha,j}_{d\beta}$ for which we do not have access to pointwise estimates. Thus, we have 
$$
| P^{\alpha,j}_{d\beta}(\bar{\varphi}) Y^\beta(g) | \le (1+(\log(1+t))^{d-1} \sum_{|\eta|+|\beta| \le |\alpha|}\sum_{\varphi \in \gamma_m} |Y^{\eta}(\varphi)| \, |Y^\beta(g) |.
$$

Now in the above sum, either $|\eta| > N-(9/2+2)$, in which case $|\beta| \le 9/2+3$ and we apply the product estimates of Proposition \ref{prop:vwpe}, or $|\eta| \le N-(9/2+2)$, in which case we can still estimate $|Y^{\eta}(\varphi)|$ pointwise. In this case, we use the $v$-weighted bounds on $Y^\beta(f)$ contain in the norm $E_N$, see \eqref{eq:tn}. In both case, we can absorb the $t$-growth thanks to the $t$ weights in the definition of $J_1$.
\item Estimates on $J_2$. These are obtained similarly, using Proposition \ref{prop:perg} instead of \ref{prop:vwpe}, since there is no $t$ weight in $J_2$ to absorb any $t$ growth. 




\end{enumerate}

\end{proof}
The preceding lemma improves the boostrap assumption \eqref{3dbasp2d}. 


Similarly to Lemma \ref{lem:l2dp} and Corollary \ref{lem:dgpp}, the pointwise estimates on $\rho\left( |Y^\beta(f)| \right)$ can be transformed into pointwise estimates for $\nabla Z^\alpha \phi$. 
\begin{lemma} For all multi-index $\alpha$ such that $|\alpha| \le N-3$
$$
||\nabla Z^\alpha \phi ||_{L^2(\mathbb{R}^n)} \lesssim  \frac{\epsilon}{t^{1/2}}, $$
and for any multi-index $|\alpha| \le N-(9/2+1)$ and $Z^\alpha \in \Gamma_{s}^{|\alpha|}$, we have for all $t \in [0,T],$
$$
\left| \nabla Z^\alpha \phi \right| \lesssim \frac{\epsilon}{(1+|x|+t)^{3/2}t^{1/2}}.
$$
\end{lemma}
\begin{proof}
We use again the commutation formula \eqref{eq:cdz}. Apart from the top order terms, we can estimate all the quantities on the right-hand side using the pointwise estimates \eqref{bas2} on $Y^\rho(\varphi)$ for $|\rho| \le N-(9/2+2)$, \eqref{bas3} on $Y^\rho(\partial \varphi)$ for $|\rho| \le N-(9/2+3)$ and Proposition \ref{es:ksmc}. Once we have access to pointwise estimates, we can just follow the strategy of the proof of \ref{lem:l2dp}. Thus, the only difficult term are thos containing top order terms in $Y^\rho(\varphi)$. Those coming from the $P$ multilinear forms are of the type
$$
\frac{1}{t} \rho\left (Y^\rho(\varphi) f \right)
$$
with $|\rho|$ the largest integer such that $|\rho| \le N-(9/2+1)$, i.e. $|\rho|=N-6$. Now, since $|\rho| \le N-6$, we can consider $Y^\beta (Y^\rho(\varphi) f)$ for $|\beta| \le 3$. Thus, we may apply Proposition \ref{es:ksmc} directly to the product $Y^\rho(\varphi) f$ and we find 
$$
\left| \frac{1}{t} \rho\left (Y^\rho(\varphi) f \right) \right| \lesssim \frac{\epsilon}{ t^{1-\sigma}\left(1+t+|x|\right)^3},
$$
where we have used Proposition \ref{prop:ep} to bound the norms appearing on the right-hand side after applications of Proposition \ref{es:ksmc}. Choosing $\sigma < 1$, these terms therefore decay better than what is needed for the statement of the Lemma.

The top order terms coming from the $Q$ forms are of type
$$
\rho\left (Y^\rho(\partial_{x^i} \varphi) f \right),
$$
with $|\rho| \le N-(9/2+2)$, i.e. $|\rho| \le N-5$. We can then proceed similarly. Since we have no extra $t$ decay in front of the $Q$ forms, it is important not to lose any $t$ decay here and thus we use the improvements of Proposition \ref{prop:perg}. 
The rest of the proof is identical to that of Lemma \ref{lem:l2dp} and therefore omitted.
\end{proof}
The preceeding lemma improves \eqref{3dbasp}. To improve \eqref{3dbaspy} is then suffices to use Lemma \ref{lem:tynzp} as well as the pointwise bounds on $Y^{\alpha}(\varphi)$ \eqref{bas2}. Thus, it remains only to improve \eqref{3dbaseb}, \eqref{bas2} and \eqref{bas3}. 

\subsubsection{Improving the global bounds}
We have
\begin{lemma}For all $t \in [0,T]$,
$$
E_{N,\delta}[f(t)] \le 3/2 \epsilon,
$$
which improves \eqref{3dbaseb}.
\end{lemma}
\begin{proof}
Using the commutation formula \eqref{eq:gcfm}, we have 
$$
[T_\phi,Y^\alpha ](f)= \sum_{d=0}^{|\alpha|+1}\sum_{i=1}^n \sum_{\substack{|\gamma| \le |\alpha|,\\ |\beta| \le |\alpha|}} P^{\alpha,i}_{d \gamma \beta}(\bar{\varphi}) \partial_{x^i} Z^\gamma(\phi) Y^\beta(f),
$$
where the $P^{\alpha,i}_{d \gamma \beta}(\bar{\varphi})$ are multilinear forms of degree $d$ and signature less than $k$ such that $k \le |\alpha|-1 \le N-1$ and $k+|\gamma|+|\beta| \le |\alpha|+1 \le N+1$.
When $k \le N- (9/2+2)$, we can bound $P^{\alpha,i}_{d \gamma \beta}(\bar{\varphi})$ by a polynomial power of $1+\log(1+t)$ and we can treat the other terms as in Section \ref{se:il1zaf}.
Otherwise $k > N- (9/2+2)$ and thus $|\beta| \le 9/2+3$ and we can use the estimates on products of Proposition \ref{prop:ep}. As in Section \ref{se:igb}, it follows that all terms are integrable in $L^1\left(\mathbb{R}_t; L^1( \mathbb{R}^n_x \times \mathbb{R}_v^n )\right)$ and the lemma follows from Lemma \ref{lem:cl}.
\end{proof}

\subsubsection{Improving the pointwise bounds on $Y^\alpha(\varphi)$ and $Y^\alpha(\partial_{x^i} \varphi)$}
Finally, we conclude the proof of the $3$d case by improving \eqref{bas2} and \eqref{bas3}.
\begin{lemma}For all $\varphi \in \mathcal{M}$ and for all $t \in [0,T]$,
\begin{enumerate}
\item for all multi-index $\alpha$ with $|\alpha|\le N-(9/2+2)$,
\begin{equation} \label{ibas2}
\left| Y^\alpha( \varphi)\right| \lesssim \epsilon \left(1+\log(1+t) \right),
\end{equation}
\item for all multi-index $\alpha$ with $|\alpha|\le N-(9/2+3)$ and all $1 \le i \le n$,
\begin{equation} \label{ibas3}
\left| Y^\alpha(\partial_{x^i} \varphi)\right| \lesssim \epsilon.
\end{equation}

\end{enumerate}
\end{lemma}
\begin{proof}
Using the commutation formula \eqref{eq:gcfm}, we have 
$$
[T_\phi,Y^\alpha ](\varphi)= \sum_{d=0}^{|\alpha|+1}\sum_{i=1}^n \sum_{\substack{|\gamma| \le |\alpha|,\\ |\beta| \le |\alpha|}} P^{\alpha,i}_{d \gamma \beta}(\bar{\varphi}) \partial_{x^i} Z^\gamma(\phi) Y^\beta(\varphi),
$$
where the $P^{\alpha,i}_{d \gamma \beta}(\bar{\varphi})$ are multilinear forms of degree $d$ and signature less than $k$ such that $k \le |\alpha|-1 \le N-1$ and $k+|\gamma|+|\beta| \le |\alpha|+1 \le N+1$. Given the range of the indices, we can estimates all terms on the right-hand side pointwise and find that 
$$
\left| [T_\phi,Y^\alpha ](\varphi) \right| \lesssim \frac{\left( 1+\log(1+t) \right)^{|\alpha|+2} \epsilon }{t^2},
$$
which is integrable in $t$. 
On the other hand, we have $Y^\alpha T_\phi(\varphi)=Y^\alpha t \nabla Z(\phi).$ Using Lemma \ref{lem:tynzp}, it follows that 
$$Y^\alpha T_\phi(\varphi)=t Z^\alpha \nabla Z \phi +\sum_{d=1}^{|\alpha|} P^\alpha_{d\beta} ( \bar{\varphi} ) Z^\beta \nabla Z \phi,$$
where the second term is integrable in $t$, while the first term satisfied only the weak bound, for all 
$$
| t Z^\alpha \nabla Z \phi | \lesssim \frac{\epsilon}{1+t}.
$$
Since any solution to $T_\phi(g)=h$ with $0$ initial data satisfies, for all $t > 0$, and for all $x,v \in \mathbb{R}^n_x \times \mathbb{R}^n_v$, $$|g|(t,x,v) \lesssim \int_0^t \sup_{(x',v')\in \mathbb{R}^n_x \times \mathbb{R}^n_v} |h(s,x,v)|ds,$$ \eqref{ibas2} follows.  \eqref{ibas3} can be obtained simarly, using that 
$$
| t Z^\alpha \nabla Z \partial_{x^i} \phi | \lesssim \frac{\epsilon}{t^2}.
$$
since $\partial_{x^i} \phi= \frac{1}{t} \left( t \partial_{x^i} \right)(\phi)$ with $t \partial_{x^i} \in \Gamma$.
\end{proof}

\bibliographystyle{hacm}
\bibliography{refs}

\end{document}